\documentclass[a4paper,leqno]{amsart}

\usepackage{latexsym}
\usepackage[english]{babel}
\usepackage{fancyhdr}
\usepackage[mathscr]{eucal}
\usepackage{amsmath}
\usepackage{mathrsfs}
\usepackage{amsthm}
\usepackage{amsfonts}
\usepackage{amssymb}
\usepackage{amscd}
\usepackage{bbm}
\usepackage{graphicx}
\usepackage{subcaption}
\usepackage{graphics}
\usepackage{latexsym}
\usepackage{color}

\usepackage{pifont}
\usepackage{booktabs} % for '\midrule' macro

\newcommand{\ud}{\mathrm{d}}

\newcommand{\ii}{\mathrm{i}}
\newcommand{\cH}{\mathcal{H}}

\newcommand{\C}{\mathbb C}

\theoremstyle{plain}
\newtheorem{theorem}{Theorem}[section]
\newtheorem{lemma}[theorem]{Lemma}

\newtheorem{proposition}[theorem]{Proposition}

\theoremstyle{definition}

\newtheorem{remark}[theorem]{Remark}
\newtheorem{example}[theorem]{Example}

\numberwithin{equation}{section}

\begin{document}

\title[Krylov solvability under perturbations]
{Krylov solvability under perturbations of abstract inverse linear problems}

\author[No\`e Angelo Caruso]{No\`e Angelo Caruso}

\address[N.~A.~Caruso]{%
% International School for Advanced Studies\\
% via Bonomea 265, I-34136 Trieste (ITALY)\\
%and 
Gran Sasso Science Institute\\
Viale F.~Crispi 7, I-67100 L'Aquila (ITALY)}

\email{noe.caruso@gssi.it}

\author[Alessandro Michelangeli]{Alessandro Michelangeli}
\address[A.~Michelangeli]{Institute for Applied Mathematics, and Hausdorff Center of Mathematics, University of Bonn \\ Endenicher Allee 60 \\ 
D-53115 Bonn (GERMANY).}
\email{michelangeli@iam.uni-bonn.de}

%\dedicatory{}

\begin{abstract}
 When a solution to an abstract inverse linear problem on Hilbert space is approximable by finite linear combinations of vectors from the cyclic subspace associated with the datum and with the linear operator of the problem, the solution  is said to be a Krylov solution, i.e., it belongs to the Krylov subspace of the problem. Krylov solvability of the inverse problem allows for solution approximations that, in applications, correspond to the very efficient and popular Krylov subspace methods. We study here the possible behaviours of persistence, gain, or loss of Krylov solvability under suitable small perturbations of the inverse problem -- the underlying motivations being the stability or instability of Krylov methods under small noise or uncertainties, as well as the possibility to decide a priori whether an inverse problem is Krylov solvable by investigating a potentially easier, perturbed problem. We present a whole scenario of occurrences in the first part of the work. In the second, we exploit the weak gap metric induced, in the sense of Hausdorff distance, by the Hilbert weak topology, in order to conveniently monitor the distance between perturbed and unperturbed Krylov subspaces.
\end{abstract}

\date{\today}

\subjclass[2020]{}

\thanks{This work is partially supported by the Alexander von Humboldt Foundation.}

\maketitle

%\tableofcontents

\section{Introduction}\label{intro}

The ubiquitous occurrence of linear phenomena that produce an output $g$ from an input $f$ according to a linear law $A$, so that from the exact or approximate measurement of $g$ one tries to recover exact or approximate information on $f$, gives rise to the following abstraction. The possible inputs and outputs form an abstract linear vector space $\cH$ on which an action is performed by a linear operator $A$, and given a datum $g\in\mathrm{ran}A$ one searches for solution(s) to the inverse linear problem $Af=g$. In fact, a vast variety of phenomena are encompassed by such a mathematical generalisation when $\cH$ is taken to be a (possibly infinite-dimensional) inner product and complete complex vector space, namely a complex Hilbert space, and $A$ is a closed linear operator acting on $\cH$. The (somewhat minimal) requirement of operator closedness is aimed at having a non-trivial notion of spectrum of $A$, hence to allow for the possible use of spectral methods in solving $Af=g$. The primary interest in such setting is to obtain convenient approximations of $f$ in terms of approximants produced by certain algorithms.

Here various levels of abstraction are implemented: (a) the dimensionality of $\cH$, finite or infinite, (b) the boundedness or unboundedness of $A$, (c) the spectral properties of $A$ (from a purely discrete spectrum to richer structures with continuous components, separated or not from zero, and so on). When $\dim\cH<\infty$ the inverse problem involves finite matrices and is typically under a very accurate control in all its aspects (algebraic, analytical, numerical, including in applications the control of the rate of convergence of approximants, etc.). To a lesser degree one has a theory of inverse linear problems governed by bounded operators on infinite-dimensional spaces: it consists of a more limited amount of sophisticated results, only few of which have a counterpart in the unbounded case. The literature is obviously enormous: we refer to that body of ideas and tools generally called iterative methods \cite{Saad-2003_IterativeMethods}, Petrov-Galerkin methods \cite{Ern-Guermond_book_FiniteElements,Quarteroni-book_NumModelsDiffProb}, generalised projection methods \cite{CMN-truncation-2018}, Krylov projection methods \cite{Saad-2003_IterativeMethods,Liesen-Strakos-2003}, in which as said the complex of knowledges when $A$ is bounded on an infinite-dimensional $\cH$ is certainly less systematic,
%-- it is worth mentioning results such as \cite{Karush-1952,Daniel-1967,Kammerer-Nashed-1972,Nemirovskiy-Polyak-1985,Nemirovskiy-Polyak-1985-II,Winther-1980,Herzog-Ekkehard-2015,CMN-2018_Krylov-solvability-bdd} -- 
let alone when $A$ itself is unbounded.
%\cite{CM-Nemi-unbdd-2019,CM-2019_ubddKrylov}.

With a clear motivation from applications, in the abstract problem outlined above it is relevant to investigate when the solution $f$ admits a subspace of distinguished approximants in $\cH$, explicitly constructed from $A$ and $g$ as finite linear combinations of $g,Ag,A^2g,$ $A^3g,\dots$. That is, one introduces the `\emph{Krylov subspace}'
\begin{equation}\label{eq:defKrylov}
  \mathcal{K}(A,g) \;:=\;  \mathrm{span} \{ A^kg\,|\,k\in\mathbb{N}_0\}\;\subset\;\cH\,,
\end{equation}
and inquires whether $f\in\overline{\mathcal{K}(A,g)}$ (the closure of $\mathcal{K}(A,g)$ in the norm topology of $\cH$). When this is the case, the problem $Af=g$ is said to be `\emph{Krylov solvable}', and one refers to the solution(s) $f$ as `\emph{Krylov solution(s)}'.

Let us stress that for unbounded $A$ the notion of Krylov subspace only makes sense if the datum $g$ is `$A$-\emph{smooth}', meaning $g\in C^\infty(A)$ with
\begin{equation}\label{eq:gsmooth}
 %6f^{[0]}
 C^\infty(A)\;:=\;\bigcap_{N\in\mathbb{N}}\mathcal{D}(A^N)\,,
\end{equation}
where $\mathcal{D}(\cdot)$ is the notation for the operator domain (in applications where $A$ is a differential operator, $g\in C^\infty(A)$ is a regularity requirement); $A$-smoothness is automatic when $A$ is everywhere defined and bounded on $\cH$. Let us also observe that occurrence $\overline{\mathcal{K}(A,g)}=\cH$ (that makes the Krylov solvability question trivial) corresponds to the fact that $g$ is a cyclic vector for $A$. Noticeably, the set of cyclic vectors for a bounded operator is either empty or dense in $\cH$ \cite{Geher-1972}, and it is unknown whether there exists a bounded operator on a separable Hilbert space $\cH$ such that \emph{every} non-zero vector in $\cH$ is cyclic. A prototypical mechanism for $\overline{\mathcal{K}(A,g)}$ to only be a proper closed subspace of $\cH$ is provided by the right shift operator $R$ on $\ell^2(\mathbb{N}_0)$ (defined as usual by $R e_n=e_{n+1}$ on the canonical basis): for instance, $\overline{\mathcal{K}(R,e_1)}$ is the orthogonal complement to the span of $e_0$. Another such mechanism is when $A$ is reduced with respect to the Hilbert space decomposition $\cH=\cH_1\oplus\cH_2$ and $g\in\cH_1$.

Again, it is no surprise that Krylov solvability is well understood in finite dimensions \cite{Saad-2003_IterativeMethods,Liesen-Strakos-2003}, with instead only partial results for bounded $A$ in infinite dimensions \cite{Karush-1952,Daniel-1967,Kammerer-Nashed-1972,Nemirovskiy-Polyak-1985,Nemirovskiy-Polyak-1985-II,Winther-1980,Herzog-Ekkehard-2015,CMN-2018_Krylov-solvability-bdd} or for unbounded $A$ \cite{CM-Nemi-unbdd-2019,CM-2019_ubddKrylov}.

In the bounded case it is worth recalling a few facts we recently established in \cite{CMN-2018_Krylov-solvability-bdd}. As customary, let us denote by $\mathcal{B}(\cH)$ the algebra of operators that are everywhere defined and bounded on $\cH$, equipped with the usual operator norm $\|A\|_{\mathrm{op}}$. Thus, let now $A\in\mathcal{B}(\cH)$ and $g\in\mathrm{ran}A$.
\begin{itemize}
 \item[(I)] If $A$ is reduced with respect to the `\emph{Krylov decomposition}'
 \[
  \cH\;=\;\overline{\mathcal{K}(A,g)}\oplus\mathcal{K}(A,g)^\perp
 \]
 (for short, if $A$ is `$\mathcal{K}(A,g)$-\emph{Krylov-reducible}'), then there exists a Krylov solution to $Af=g$.
 \item[(II)] If the `\emph{Krylov intersection}' subspace
 \[
  \mathcal{I}(A,g)\;:=\;\overline{\mathcal{K}(A,g)}\cap A(\mathcal{K}(A,g)^\perp)
 \]
 is the trivial set $\{0\}$, then there exists a Krylov solution to $Af=g$.
\end{itemize}
 Thus, (I) and (II) provide mechanisms for Krylov solvability, the second being in fact more general (Krylov reducibility always implies triviality of the Krylov intersection, but not the other way around, in general). (II) is in a sense \emph{the} intrinsic mechanism under the additional condition that $A$ has everywhere defined and bounded inverse, for in this case Krylov solvability is \emph{equivalent} to the $\mathcal{I}(A,g)$-triviality. Furthermore:	
\begin{itemize}
 \item[(III)] If $A$ is normal (or, more generally, if $\ker A\subset\ker A^*$), then the Krylov solution to $Af=g$, if existing, is unique.
 \item[(IV)] If $A$ is self-adjoint, then the problem $Af=g$ admits a unique Krylov solution $f$.
\end{itemize}
 Property (IV) when in particular $A$ is positive definite was previously established by Nemirovskiy and Polyak \cite{Nemirovskiy-Polyak-1985}, by showing that the sequence of Krylov approximants obtained by the conjugate gradient algorithm converges strongly to the exact solution. Further examples and classes of operators in $\mathcal{B}(\cH)$ giving rise to Krylov solvable problems were discussed in \cite{CMN-2018_Krylov-solvability-bdd}.

 To come the object of the present work, let us stress that the picture outlined so far concerns inverse linear problems where both the datum $g$ and the linear operator $A$ are known \emph{exactly} -- our analysis here is abstract operator-theoretic in nature, but with reference to the initial motivation, it is as if the law $A$ is precisely understood and the output $g$ is measured with full precision. A more general perspective is to allow for amount of uncertainty affecting the knowledge of $A$, or $g$, or both. Or, from another point of view, instead of only focusing on the inverse problem of interest, one may consider also an auxiliary, possibly more tractable problem, close in some sense to the original one, which allows for useful approximate information.

 Thus, in this work we consider \emph{perturbations} of the original problem $Af=g$ of the form $A'f'=g'$, where $A$ and $A'$, as well as $g$ and $g'$ are close in a controlled sense, and we study the \emph{effect of the perturbation on the Krylov solvability}.

 This context is clearly connected with the general framework of ``ill-posed'' inverse linear problems \cite{Hanke-ConjGrad-1995,Hansen-Illposed-1998}, where only the perturbed quantities $A'$ or $g'$ are accessible, due for instance to measurement errors, and ill-posedness manifests for instance through the fact that $g'\notin\mathrm{ran}A$, the goal being to approximate the actual solution $f$ in a controlled sense.

 Yet, the questions that we intend to address have a different spirit. We keep regarding $A$ and $g$ as exactly known or, in principle, exactly accessible, but with the idea that close to the problem $Af=g$ there is a perturbed problem $A'f'=g'$ that serves as an auxiliary one, possibly more easily tractable, say, with Krylov subspace methods, in order to obtain conclusions on the Krylov solvability of the original problem. Or, conversely, we inquire under which conditions the nice property of Krylov solvability for $Af=g$ is stable enough to survive a small perturbation (that in applications could arise, again, from experimental or numerical uncertainties), or when instead Krylov solvability is washed out by even small inaccuracies in the precise knowledge of $A$ or $g$  -- an occurrence in which Krylov subspace methods would prove to be unstable. And, more abstractly, we pose the question of a convenient notion of vicinity between the subspaces $\overline{\mathcal{K}(A,g)}$ and $\overline{\mathcal{K}(A',g')}$ when $A$ and $A'$ (respectively, $g$ and $g'$) are suitably close.

 In Section \ref{sec:questions} we elaborate more diffusely on this body of questions and their conceptual relevance of the present abstract setting.

There exists a large amount of literature that accounts for perturbations in Krylov subspace methods, most of which approaches the problem from the point of view of inexact Krylov methods (see, e.g., \cite{Zemke-2007,Vandeneshof-2005,Simoncini-Szyld-inexact-2005,Simoncini-Szyld-inexact-2003,Sifuentes-Embree-Morgan-2013,Xue-Elman-2011,Du-Sarkis-Schaerer-Szyld-2013}). A good outline of the theory of inexact Krylov methods may be found in \cite{Simoncini-Szyld-inexact-2003}, and see in particular \cite{Zemke-2007} for a general setting of Krylov algorithms under perturbations.  
The idea underlying inexact Krylov methods is that the exact (typically non-singular) inverse linear problem in $\C^N$, $Af = g$, is perturbed in $A$ by a series of linear operators $(E_k)_{k=1}^N$ on $\C^N$ that may change at each step $k$ of the algorithm. Typical scenarios that could induce such perturbations at each step of the algorithm are, but not limited to, truncation and rounding errors in finite precision machines, or approximation errors from calculating complicated matrix-vector products.
The main results reported in \cite{Zemke-2007,Vandeneshof-2005,Simoncini-Szyld-inexact-2005,Simoncini-Szyld-inexact-2003,Sifuentes-Embree-Morgan-2013,Xue-Elman-2011,Du-Sarkis-Schaerer-Szyld-2013} include the convergence behaviour of the error and residual terms, in particular their rates, and typically bounds on how far these indicators of convergence are from the unperturbed setting at a given iteration number $k$. 

Yet, these investigations are of a different nature than what we propose here. To begin with, the typical analysis of inexact Krylov methods is in the finite-dimensional setting, where we already have a good control over the Krylov-solvability of inverse problems as well as the rates of convergence to a solution. Furthermore, it is not discussed in this literature how the underlying Krylov subspaces themselves change, as well as the richer phenomena pertaining to the Krylov-solvability (or lack-of) of the underlying problem and its perturbations; some of the very questions we are interested in investigating.
 
 The point is that, to our knowledge, this line of investigation is so far essentially uncharted. With this spirit, and in view of the set of general questions outlined in Section \ref{sec:questions}, in Section \ref{sec:gain-loss} we present an overview of typical phenomena that may occur to the Krylov solvability of an inverse linear problem $Af=g$ in terms of the Krylov solvability, or lack of thereof, of auxiliary inverse problems where $A$ or $g$ or both are perturbed in a controlled sense. Such survey indicates that the sole control of the operator or of the data perturbation, in the respective operator and Hilbert norm, still leaves the possibility open to all phenomena such as the persistence, gain, or loss of Krylov solvability in the limit $A_n\to A$ or $g_n\to g$, where $A_n$ (and so $g_n$) is the generic element of a sequence of perturbed objects. The implicit explanation is that an information like $A_n\to A$ or $g_n\to g$ is not enough to account for a suitable vicinity of the corresponding Krylov subspaces -- we discussed in points (I) and (II) above that the Krylov solvability of the inverse problem $Af=g$ corresponds to certain structural properties of the subspace $\mathcal{K}(A,g)$, therefore one implicitly needs to monitor how the latter properties are preserved or altered under the perturbation. This also suggests that the additional constraint of performing the perturbation within certain subclasses of operators may supplement further information on Krylov solvability: this is in principle a vast programme, in Section \ref{sec:K-class} we focus on the operators of $\mathscr{K}$-class we had previously considered in \cite{CMN-2018_Krylov-solvability-bdd}, and discuss the robustness and fragility of this class from the perturbative perspective of the induced inverse problems.

 In the second part of this work, Sections \ref{sec:weakgapmetric}-\ref{sec:Kry-perturb-and-wgmetric}, we address more systematically the issue of vicinity of Krylov subspaces in a sense that be informative for the Krylov solvability of the corresponding inverse problems. What shows encouraging properties, next to some serious limitations, though, is the comparison of (the closures of) two Krylov subspaces in terms of the Hausdorff distance between the respective unit balls, considered as closed subset of the Hilbert unit ball when the latter is metrised with respect to the \emph{weak} Hilbert topology. (Had we used the \emph{norm} topology, that would have not even controlled the very intuitive convergence of the finite-dimensional Krylov subspaces, namely with iterates up to some $A^{N_0}g$, to its infinite-dimensional counterpart, as $N_0\to\infty$.) This framework leads to appealing approximation results, as the inner approximability of Krylov subspaces established in Subsect.~\ref{sec:innerapprox}. Right after, Proposition \ref{prop:Krisolv-along-onelimit} is a prototype of the kind of perturbative results we had originally in mind, namely a control of the perturbation, formulated in terms of the perturbed and unperturbed Krylov subspaces, that \emph{predicts} the persistence of Krylov solvability when the perturbation is removed.

 In this spirit we rather intended -- and in the above sense managed -- to open a perspective on a general problem, essentially not addressed so far, that is operator-theoretic in nature, yet with direct motivations from Krylov approximation algorithms in numerical computation. The corpus of partial results that we present here only scratch the surface of a problem that in our intentions need be further investigated. We shall collect more explicit conclusions in this sense in the final Section \ref{sec:conclusions}.

 \medskip

 \textbf{Notation.} Besides further notation that will be declared in due time, we shall keep the following convention. $\cH$ denotes a complex Hilbert space with norm $\|\cdot\|$ and scalar product $\langle\cdot,\cdot\rangle$, anti-linear in the first entry and linear in the second. Norm and weak convergence in $\cH$ are denoted, as usual, with $x_n\to x$ and $x_n\rightharpoonup x$.  $\mathcal{B}(\cH)$ is the complete, norm, $*$-algebra of everywhere defined and bounded linear operators on $\cH$, equipped with customary operator norm $\|\cdot\|_{\mathrm{op}}$. We shall often omit the adjective `linear', with reference to operators. $\mathbbm{1}$ and $\mathbbm{O}$ denote, respectively, the identity and the zero operator. $\sigma(A)$ denotes the spectrum of some (closed) linear operator $A$ on $\cH$. 
 $\overline{\mathcal{V}}$ is the norm closure of the span of the vectors in $\mathcal{V}$ when $\mathcal{V}$ is a subset of $\cH$, and $\mathcal{V}^\perp$ is the largest closed subspace of $\cH$ whose vectors are orthogonal to all elements of the subset $\mathcal{V}\subset\cH$. For $\psi,\varphi\in\cH$, by $|\psi\rangle\langle\psi|$ and $|\psi\rangle\langle\varphi|$ we shall denote the $\cH\to\cH$ rank-one maps acting respectively as $f\mapsto \langle \psi, f\rangle\,\psi$ and $f\mapsto \langle \varphi, f\rangle\,\psi$ on generic $f\in\cH$.

 \section{Krylov solvability from a perturbative perspective}\label{sec:questions}

 As argued already in the Introduction, the question of the effects of perturbations on the Krylov solvability of an infinite-dimensional inverse linear problem is essentially new.

 One easily realises that such question takes a multitude of related, yet somewhat different formulations depending on the precise perspective one looks at it. Given the essential novelty of this line of investigation, we find it instructive to organise the most relevant of such queries into a coherent scheme -- which is the goal of this Section. This serves both as a reference for the results and explicit partial answers that we  give in this work, as well as an ideal road map for future studies.

 In practice, let us discuss the following main categories of connected problems. For the first three of them, we work in the bounded case, with $\cH$ being the underlying complex infinite-dimensional Hilbert space.

 \textbf{I. Comparison between ``close'' Krylov subspaces.} This is the abstract problem of providing a meaningful comparison between $\overline{\mathcal{K}(A,g)}$ and $\overline{\mathcal{K}(A',g')}$, as two closed subspaces of $\cH$, for given $A,A'\in\mathcal{B}(\cH)$ and $g,g'\in\cH$ such that in some convenient sense $A$ and $A'$, as well as $g$ and $g'$ are close. As a priori such subspaces might only have a trivial intersection, the framework is rather that of comparison of subspaces of a normed space, in practice introducing convenient topologies or metric distances.
 %($\rightarrow$ \textbf{vague, to elaborate more + references}). 

 A more application-oriented version of the same problem is the following. Given $A\in\mathcal{B}(\cH)$ and $g\in\cH$, one considers approximants of one or the other (or both), say, sequences $(A_n)_{n\in\mathbb{N}}$ and $(g_n)_{n\in\mathbb{N}}$ respectively in $\mathcal{B}(\cH)$ and $\cH$, such that $\|A_n-A\|_{\mathrm{op}}\to 0$ and $\|g_n-g\|\to 0$ as $n\to\infty$. Then the question is whether a meaningful notion of limit $\overline{\mathcal{K}(A_n,g_n)}\to\overline{\mathcal{K}(A,g)}$ can be defined.

 \textbf{II. Perturbations preserving/creating Krylov solvability.} This question is inspired to the possibility that, given $A\in\mathcal{B}(\cH)$ and $g\in\mathrm{ran} A$, instead of solving the ``difficult'' inverse problem $Af=g$ one solves a convenient perturbed problem $A'f'=g'$, with $A'\in\mathcal{B}(\cH)$ and $g'\in\mathrm{ran} A'$ close respectively to $A$ and $g$, which is ``easily'' Krylov solvable, and the Krylov solution of which provides approximate information to the original solution $f$.
 
 Here is an explicit set-up for this question. Assume that one finds $(A_n)_{n\in\mathbb{N}}$ in $\mathcal{B}(\cH)$ and $(g_n)_{n\in\mathbb{N}}$ in $\cH$ such that the inverse problems $A_nf_n=g_n$ are all Krylov solvable and 
 $\|A_n-A\|_{\mathrm{op}}\to 0$ and $\|g_n-g\|\to 0$ as $n\to\infty$. Is $Af=g$ Krylov solvable too? And if at each perturbed level $n$ there is a unique Krylov solution $f_n$, does one have $\|f_n-f\|\to 0$ where $f$ is a (Krylov) solution to $Af=g$? 
 
 One scenario of applications is that for $A_nf_n=g_n$ Krylov solvability comes with a much more easily (say, faster) solvable solution algorithm, so that $f$ is rather determined as $f=\lim_{n\to\infty}f_n$ instead of directly approaching the problem $Af=g$. 
 
 Another equally relevant scenario is that the possible Krylov solvability of the problem of interest $Af=g$ is initially \emph{unknown}, and prior to launching resource-consuming Krylov algorithms for solving the problem, one wants to be guaranteed that a Krylov solution indeed exists. To this aim one checks the Krylov solvability for $A_n f_n =g_n$ uniformly in $n$, and the convergence $A_n\to A$, $g_n\to g$, thus coming to an affirmative answer.

 \textbf{III. Perturbations destroying Krylov solvability.}  The opposite occurrence has to be monitored as well, namely the possibility that a small perturbation of the Krylov solvable problem $Af=g$ produces a non-Krylov solvable problem $A'f'=g'$. Say, if in the above setting \emph{none} of the problems $A_nf_n=g_n$ are Krylov solvable and yet $A_n\to A$ and $g_n\to g$, under what conditions does one gain Krylov solvability in the limit for the problem $Af=g$? A comprehension of this phenomenon would be of great relevance to identify those circumstances when Krylov methods are intrinsically unstable, in the sense that even a tiny uncertainty in the knowledge of $A$ and/or $g$ brings to a perturbed problem $A'f'=g'$ for which, unlike the exact problem of interest $Af=g$, Krylov methods are not applicable.

 In the \emph{unbounded} case, more precisely when the operator $A$ is closed and unbounded on $\cH$, in principle all the above questions have their own counterpart, except that the fundamental condition $g\in C^\infty(A)$ required to have a meaningful notion of $\mathcal{K}(A,g)$ is highly unstable under perturbations, and one has to ensure case by case that certain problems are well posed.

 Yet, for its evident relevance let us highlight the following additional class of questions.

 \textbf{IV. Perturbations-regularisations exploiting Krylov solvability.} For the problem of interest $Af=g$ one might well have $g\in\mathrm{ran}A$ \emph{but} $g\notin C^\infty(A)$. In this case Krylov methods are not applicable: there is no actual notion of Krylov subspace associated to $A$ and $g$, hence no actual Krylov approximants to utilize iteratively. Assume though that one finds a sequence $(g_n)_{n\in\mathbb{N}}$ entirely in $\mathrm{ran}A\cap C^\infty(A)$ with $\|g_n-g\|\to 0$. This occurrence is quite typical: if $A$ is a differential operator on $L^2(\mathbb{R}^d)$, everyone is familiar with sequences $(g_n)_{n\in\mathbb{N}}$ of functions that all have high regularity, uniformly in $n$, and for which the $L^2$-limit $g_n\to g$ produces a rough function $g$. Assume further that each problem  $Af_n=g_n$ is Krylov solvable. For example, as we showed in \cite[Theorem 4.1]{CM-2019_ubddKrylov}, for a vast class of self-adjoint or skew-adjoint $A$'s, possibly unbounded, there exists a unique solution $f_n\in\overline{\mathcal{K}(A,g_n)}$. This brings the following questions. First, do the $f_n$'s have a limit $f$ and does $f$ solve $Af=g$? And, more abstractly speaking, is there a meaningful notion of the limit $\lim_{n\to\infty}\overline{\mathcal{K}(A,g_n)}$, \emph{irrespectively} of the approximant sequence $(g_n)_{n\in\mathbb{N}}$, that could then be interpreted as a replacement for the non-existing Krylov subspace associated to $A$ and $g$? The elements of such limit subspace would provide exploitable approximants for the solution to the original problem $Af=g$.

 One last remark concerns the topologies underlying all the questions above. We explicitly formulated them in terms of the operator norm and Hilbert norm, but alternatively there is a variety of weaker notions of convergence that are still highly informative for the solution to the considered inverse problem -- we discussed this point extensively in 
 \cite{CMN-truncation-2018}. Thus, the ``weaker'' counterpart of the above questions represents equally challenging and potentially useful problems to address.

\section{Gain or loss of Krylov solvability under perturbations}\label{sec:gain-loss}

This Section is meant to present examples of different behaviours that may occur in those cases belonging to the categories II and III contemplated in the previous Section. In practice we are comparing here the ``unperturbed'' inverse linear problem $Af=g$ with ``perturbed'' problems of the form $Af_n=g_n$, or $A_nf_n=g$, along a sequence $(g_n)_{n\in\mathbb{N}}$ such that $\|g_n-g\|\to 0$, or along a sequence $(A_n)_{n\in\mathbb{N}}$ such that $\|A_n-A\|_{\mathrm{op}}\to 0$. Our particular focus is the Krylov solvability, namely its preservation, or gain, or loss in the limit $n\to\infty$. Next to operator perturbations ($A_n\to A$) and data perturbations ($g_n\to g$), it is also natural to consider simultaneous perturbations of the both of them.

The purpose here is two-fold: we want to convey a concrete flavour of how inverse problems behave under controlled perturbations of the operator or of the datum, as far as having Krylov solutions is concerned, and we also want to highlight the emerging, fundamental, and a priori unexpected lesson. Which is going to be, in short: the sole control that $A_n\to A$ or $g_n\to g$ is not enough to predict whether Krylov solvability is preserved, or gained, or lost in the limit, i.e., each such behaviour can actually occur. The immediate corollary of this conclusion is: one must describe the perturbation of the problem $Af=g$ by means of additional information, say, by restricting to particular sub-classes of inverse problems, or by introducing suitable notions of vicinity of Krylov subspaces, in order to control the effect of the perturbation on Krylov solvability. It is this latter consideration that motivates the more specific discussion of Section \ref{sec:K-class} and of Sections \ref{sec:weakgapmetric}-\ref{sec:Kry-perturb-and-wgmetric}.

For the examples that follow we shall choose concrete playgrounds that allow for the (in general non-trivial) explicit identification of the Krylov subspace. 
\begin{itemize}
 \item As typical cases of Krylov solvable inverse problems we should have in mind, for instance, self-adjoint operators $A\in\mathcal{B}(\cH)$ with $g\in\mathrm{ran}\,A$ (\cite[Corollary 3.11]{CMN-2018_Krylov-solvability-bdd}), or the Volterra operator on $\cH=L^2[0,1]$, namely the compact, normal, linear map $V$ such that $(Vf)(x):=\int_0^xf(y)\ud y$ (in particular, $x^k\mapsto \frac{1}{k+1}x^{k+1}$ for any $k\in\mathbb{N}_0$), for which we know that $\overline{\mathcal{K}(V,g)}=\cH$ for any monomial $g=x^k$ (\cite[Example 3.1]{CMN-2018_Krylov-solvability-bdd}). One may find other possibilities discussed in our work \cite{CMN-2018_Krylov-solvability-bdd}. 
 \item Instead, as a typical source of lack of Krylov solvability (\cite[Appendix A]{CMN-2018_Krylov-solvability-bdd}) we use the right-shift operator on $\ell^2$-spaces: the basic version is on $\cH=\ell^2(\mathbb{N})$, with canonical orthonormal basis  $(e_k)_{k\in\mathbb{N}}$, where the right-shift is $R=\sum_{k\in\mathbb{N}}|e_{k+1}\rangle\langle e_k|$ (the sum converging strongly in the operator sense, $\|R\|_{\mathrm{op}}=1$, $Re_k=e_{k+1}$). Other variants are the right shift on $\ell^2(\mathbb{Z})$, or the compact counterpart  $R=\sum_{k}\lambda_{|k|}|e_{k+1}\rangle\langle e_k|$ with weights $\lambda_k>\lambda_{k+1}>0$ and $\lambda_k\xrightarrow{k\to\infty}0$ (now the series converging in operator norm). Any such $R$ admits both a dense of non-cyclic vectors, and a dense of cyclic vectors (see, e.g., \cite{Herrero-1972,Shkarin-2006-supecylic-shift}).
\end{itemize}

\subsection{Operator perturbations}~

\begin{example}
 Let $R$ be the weighted (compact) right-shift operator 
 \[
  R\;:=\;\sum_{k=1}^\infty\frac{1}{\,k^2}|e_{k+1}\rangle\langle e_k|
 \]
 on the Hilbert space $\ell^2(\mathbb{N})$, and define 
 \[
  R_n\;:=\;\sum_{k=1}^{n-1}\frac{1}{\,k^2}|e_{k+1}\rangle\langle e_k|+\frac{1}{\,n^2}|e_1\rangle\langle e_n|\,,\qquad n\in\mathbb{N}\,,\;n\geqslant 2\,.
 \]
 As
 \[
  \begin{split}
  R-R_n\;&=\;\sum_{k=n}^{\infty}\frac{1}{\,k^2}|e_{k+1}\rangle\langle e_k|-\frac{1}{\,n^2}|e_1\rangle\langle e_n|\,, \\
  \|R-R_n\|_{\mathrm{op}}\;&\leqslant\;\sum_{k=n}^{\infty}\frac{1}{\,k^2}+\frac{1}{\,n^2}\,,
  \end{split}
 \]
 then $R_n\to R$ in operator norm as $n\to\infty$. For $g:=e_2$ and any $n\geqslant 2$, the inverse problem induced by $R_n$ and with datum $g$ has unique solution $f_n:=f:=e_1$, and so does the inverse problem induced by $R$ and with the same datum, i.e., $R_nf_n=g$ and $Rf=g$. On the other hand,
 \[
  \begin{split}
   \overline{\mathcal{K}(R_n,g)}\;&=\;\mathcal{K}(R_n,g)\;=\;\mathrm{span}\{e_1,\dots,e_n\}\,, \\
   \overline{\mathcal{K}(R,g)}\;&=\;\{e_1\}^\perp\,.
  \end{split}
 \]
 Thus, the inverse problem $R_nf_n=g$ is Krylov solvable, and obviously $f_n\to f$ in norm, yet the inverse problem $Rf=g$ is not.
\end{example}

\begin{example}
 Let $R=\sum_{k=1}^\infty|e_{k+1}\rangle\langle e_k|$ be the right-shift operator on the Hilbert space $\ell^2(\mathbb{N})$, and let
 \[
  \begin{split}
   A_n\;&:=\;|e_2\rangle\langle e_2|+\frac{1}{n}R\,,\qquad n\in\mathbb{N}\,, \\
   A\;&:=\;|e_2\rangle\langle e_2|\,, \\
   g\;&:=\;e_2\,.
  \end{split}
 \]
 Clearly, $A_n\to A$  in operator norm as $n\to\infty$.
 The inverse problem $A_nf_n=g$ has unique solution $f_n=n e_1$; as $(A_n)^kg=\sum_{j=0}^k n^{-j}e_{j+2}$ for any $k\in\mathbb{N}_0$, and therefore $\overline{\mathcal{K}(A_n,g)}=\{e_1\}^\perp$, such solution is not a Krylov solution. Instead, passing to the limit, the inverse problem $Af=g$ has unique solution $f=e_2$, which is a Krylov solution since $\overline{\mathcal{K}(A,g)}=\mathrm{span}\{e_2\}$. Observe also that $f_n$ does not converge to $f$.

\end{example}

\subsection{Data perturbations}~

\begin{example}\label{ex:gainloss-with-shift}
 Let $R:\ell^2(\mathbb{Z})\to\ell^2(\mathbb{Z})$ be the usual right-shift operator. $R$ is unitary, with $R^*=R^{-1}=L$, the left-shift operator. Moreover $R$ admits a dense subset $\mathcal{C}\subset\ell^2(\mathbb{Z})$ of cyclic vectors and a dense subset $\mathcal{N}\subset\ell^2(\mathbb{Z})$, consisting of all finite linear combinations of canonical basis vectors, such that the solution $f$ to the inverse problem $Rf=g$ does not belong to $\overline{\mathcal{K}(R,g)}$. All vectors in $\mathcal{N}$ are non-cyclic for $R$.
 \begin{itemize}
  \item[(i)] (Loss of Krylov solvability.) For a datum $g\in\mathcal{N}$, the inverse problem $Rf=g$ admits a unique solution $f$, and $f$ is not a Krylov solution. Yet, by density, there exists a sequence $(g_n)_{n\in\mathbb{N}}$ in $\mathcal{C}$ with $g_n\to g$ (in $\ell^2$-norm) as $n\to\infty$, and each perturbed inverse problem $Rf_n=g_n$ is Krylov solvable with unique solution $f_n=R^{-1}g_n=Lg_n$, and with $f_n\to f$ as $n\to\infty$. Krylov solvability is lost in the limit, still with the approximant Krylov solutions converging to the solution $f$ to the original problem.
  \item[(ii)] (Gain of Krylov solvability.) For a datum $g\in\mathcal{C}$, the inverse problem $Rf=g$ is obviously Krylov solvable, as $\overline{\mathcal{K}(R,g)}=\ell^2(\mathbb{Z})$, owing to the cyclicity of $g$. Yet, by density, there exists a sequence $(g_n)_{n\in\mathbb{N}}$ in $\mathcal{N}$ with $g_n\to g$, and each perturbed inverse problem $Rf_n=g_n$ is not Krylov solvable. Krylov solvability is absent along the perturbations and only emerges in the limit, still with the solution approximation $f_n\to f$.
 \end{itemize}
\end{example}

\begin{example}\label{ex:gainloss-generic}

With respect to the Hilbert space orthogonal sum $\cH=\cH_1\oplus\cH_2$, let $A=A_1\oplus A_2$ with $A^{(j)}\in\mathcal{B}(\cH_j)$, and $g^{(j)}\in\mathrm{ran}\,A^{(j)}$, $j\in\{1,2\}$, such that the problem $A^{(1)}f^{(1)}=g^{(1)}$ is Krylov solvable in $\cH_1$, with Krylov solution $f^{(1)}$ (for instance, $\cH_1=L^2[0,1]$, $A^{(1)}=V$, the Volterra operator, $g^{(1)}=x$, $f^{(1)}=\mathbf{1}$), and the problem $A^{(2)}f^{(2)}=g^{(2)}$ is not Krylov solvable in $\cH_2$ (for instance, $\cH_2=\ell^2(\mathbb{N})$, $A^{(2)}=R$, the right shift, $g^{(2)}=e_2$, $f^{(2)}=e_1$).
  \begin{itemize}
  \item[(i)] (Lack of Krylov solvability persists in the limit.) The inverse problems $A f_n=g_n$, $n\in\mathbb{N}$, with $g_n:=\big(\frac{1}{n}g^{(1)}\big)\oplus g^{(2)}$, are all non-Krylov solvable, with solution(s) $f_n=\big(\frac{1}{n}f^{(1)}\big)\oplus f^{(2)}$. In the limit, $g_n\to g:=0\oplus g^{(2)}$ in $\cH$, whence also $f_n\to 0\oplus f^{(2)}=:f$. The inverse problem $Af=g$ has solution $f$ (modulo $\ker  A^{(1)}\oplus \{0\}$), but $f$ is not a Krylov solution.
  \item[(ii)] (Krylov solvability emerges in the limit.) The inverse problems $A f_n=g_n$, $n\in\mathbb{N}$, with $g_n:=g^{(1)}\oplus\big(\frac{1}{n}g^{(2)}\big)$, are all non-Krylov solvable, with solution(s) $f_n=f^{(1)}\oplus\big(\frac{1}{n}f^{(2)}\big)$. In the limit, $g_n\to g:=g^{(1)}\oplus 0$ in $\cH$, whence also $f_n\to f^{(1)}\oplus 0=:f$. The inverse problem $Af=g$ has solution $f$ (modulo $\{0\}\oplus\ker  A^{(2)}$), and $f$ is a Krylov solution.
  \end{itemize}
\end{example}

\subsection{Simultaneous perturbations of operator and data}~

\begin{example}\label{ex:gainloss-simultaneous}
Same setting as in Example \ref{ex:gainloss-generic}.
 \begin{itemize}
	\item[(i)] (Lack of Krylov solvability persists in the limit.) 	The inverse problems $A_nf_n=g_n$, $n\in\mathbb{N}$, with $A_n := (\frac{1}{n} A^{(1)}) \oplus A^{(2)}$ and $g_n := (\frac{1}{n} g^{(1)}) \oplus g^{(2)}$, are all non-Krylov solvable, with solutions $f_n=f^{(1)}\oplus f^{(2)}$. In the limit, $A_n \to A := \mathbbm{O} \oplus A^{(2)}$ in operator norm and $g_n \to g := 0 \oplus g^{(2)}$ in $\cH$. The inverse problem $A f = g$ has solution $f:=0\oplus f^{(2)}$ (modulo $\ker  A^{(1)}\oplus \{0\}$), which is not a Krylov solution. Moreover in general $f_n$ does not converge to $f$.
	\item[(ii)](Krylov solvability emerges in the limit.) The inverse problems $A_nf_n=g_n$, $n\in\mathbb{N}$, with $A_n := A^{(1)} \oplus (\frac{1}{n} A^{(2)})$ and $g_n := g^{(1)} \oplus (\frac{1}{n} g^{(2)})$, are all non-Krylov solvable, with solutions $f_n=f^{(1)}\oplus f^{(2)}$. In the limit, $A_n \to A := A^{(1)}\oplus \mathbbm{O}$ in operator norm and $g_n \to g := g^{(1)}\oplus 0$ in $\cH$. The inverse problem $A f = g$ has solution $f:=f^{(1)}\oplus 0$ (modulo $\{0\}\oplus\ker  A^{(2)}$), which is a Krylov solution. Moreover in general $f_n$ does not converge to $f$.
 \end{itemize} 
\end{example}

 \section{Krylov solvability along perturbations of $\mathscr{K}$-class}\label{sec:K-class}

 In a previous work \cite{CMN-2018_Krylov-solvability-bdd} we singled out a class of operators in $\mathcal{B}(\cH)$ that in retrospect display relevant behaviour as far as Krylov solvability along perturbations is concerned. In this Section we elaborate further on that class, in view of the scheme of general questions presented in Sect.~\ref{sec:questions}.

 By definition, a linear operator $A$ acting on a complex Hilbert space $\cH$ is in the $\mathscr{K}$-class when $A$ is everywhere defined and bounded, and there exists a bounded open $\mathcal{W}\subset\mathbb{C}$ containing the spectrum $\sigma(A)$ and such that $0\notin\overline{\mathcal{W}}$ and $\mathbb{C}\setminus\mathcal{W}$ is connected. In particular, a $\mathscr{K}$-class operator has everywhere defined bounded inverse.

 We have this result.
 
 \begin{theorem}\label{thm:Kclass}
  Let $A$ be a $\mathscr{K}$-class operator on a complex Hilbert space $\cH$.
  \begin{itemize}
   \item[(i)] For every $g\in\cH$ the inverse problem $Af=g$ is Krylov solvable, with unique solution $f=A^{-1}g$.
   \item[(ii)] The $\mathscr{K}$-class is open in $\mathcal{B}(\cH)$. In particular, there is $\varepsilon_A>0$ such that for any other operator $A'\in\mathcal{B}(\cH)$ with $\|A'-A\|_{\mathrm{op}}<\varepsilon_A$ the inverse problem $A'f'=g$ has a unique solution, $f'={A'}^{-1}g$, which is also a Krylov solution.
   \item[(iii)] When in addition (and without loss of generality) $\|A-A'\|_{\mathrm{op}}\leqslant(2\|A^{-1}\|_{\mathrm{op}})^{-1}$, then $f$ and $f'$ from (i) and (ii) satisfy
   \[
    \|f-f'\|\;\leqslant\;2\,\|g\|\,\|A^{-1}\|_{\mathrm{op}}^2\:\|A'-A\|_{\mathrm{op}}\,.
   \]
  \end{itemize}  
 \end{theorem}

 Theorem \ref{thm:Kclass} addresses questions of type II from the general scheme of Section \ref{sec:questions}: it provides a framework where Krylov solvability is preserved under perturbations of the linear operator inducing the inverse problem. Indeed, an obvious consequence of Theorem \ref{thm:Kclass} is: if a sequence $(A_n)_{n\in\mathbb{N}}$ in $\mathcal{B}(\cH)$ satisfies $A_n\to A$ in operator norm for some $\mathscr{K}$-class operator $A$, then eventually in $n$ the $A_n$'s are all of $\mathscr{K}$-class, the associated inverse problems $A_nf_n=g$ are Krylov solvable with unique solution $f_n=A_n^{-1}g$, and moreover $f_n\to f$ in $\cH$, where $f=A^{-1}g$ is the unique and Krylov solution to $Af=g$.

 \begin{proof}[Proof of Theorem \ref{thm:Kclass}] As we proved in \cite[Prop.~3.15]{CMN-2018_Krylov-solvability-bdd} for all $\mathscr{K}$-class operators, there exists a polynomial sequence $(p_n)_{n\in\mathbb{N}}$, consisting of polynomials in the variable $z\in\mathbb{C}$, such that $\|p_n(A)-A^{-1}\|_{\mathrm{op}}\to 0$ as $n\to\infty$. Thus, the unique solution $f$ to $Af=g$ satisfies
 \[
  \|f-p_n(A)g\|\;=\;\|A^{-1}g-p_n(A)g\|\;\leqslant\;\|g\|\,\|p_n(A)-A^{-1}\|_{\mathrm{op}}\;\xrightarrow{n\to\infty}\;0\,,
 \]
 meaning that $f\in\overline{\mathcal{K}(A,g)}$. This proves part (i).

 Concerning (ii), we use the fact that $\sigma(A)$ is an upper semi-continuous function of $A\in\mathcal{B}(\cH)$ (see, e.g., \cite[Problem 103]{Halmos-HilbertSpaceBook} and \cite[Theorem IV.3.1 and Remark IV.3.3]{Kato-perturbation}), meaning that for every bounded open set $\Omega\subset\mathbb{C}$ with $\sigma(A)\subset\Omega$ there exists $\varepsilon_A>0$ such that if $A'\in\mathcal{B}(\cH)$ with $\|A'-A\|_{\mathrm{op}}<\varepsilon_A$, then $\sigma(A')\subset\Omega$. Applying this to $\Omega=\mathcal{W}$ we deduce that any such $A'$ is again of $\mathscr{K}$-class. The remaining part of the thesis then follows from (i).

 As for (iii), clearly
 \[
 \begin{split}
  \|f-f'\|\;&\leqslant\;\|g\|\,\|A^{-1}-{A'}^{-1}\|_{\mathrm{op}}\;=\;\|g\|\,\|{A'}^{-1}(A-A')A^{-1}\|_{\mathrm{op}} \\
  &\leqslant\;\|g\|\,\|{A'}^{-1}\|_{\mathrm{op}}\:\|A^{-1}\|_{\mathrm{op}}\:\|A-A'\|_{\mathrm{op}}\,,
 \end{split}
 \]
 and 
 \[
  {A'}^{-1}\;=\;A^{-1}\sum_{n=0}^\infty\big( (A-A')A^{-1}\big)^n\quad \textrm{when }\|A-A'\|_{\mathrm{op}}<\|A^{-1}\|_{\mathrm{op}}^{-1}\,,
 \]
 whence, when additionally $\|A-A'\|_{\mathrm{op}}\leqslant\big(2\,\|A^{-1}\|_{\mathrm{op}}\big)^{-1}$, $\|{A'}^{-1}\|_{\mathrm{op}}\leqslant 2\,\|A^{-1}\|_{\mathrm{op}}$. Plugging the latter inequality into the above estimate for $\|f-f'\|$ yields the conclusion. 
 \end{proof}

 \begin{remark}
  Such `elementary' proof of Theorem \ref{thm:Kclass} relies on a non-trivial toolbox, the above-mentioned result \cite[Prop.~3.15]{CMN-2018_Krylov-solvability-bdd}.  
 \end{remark}

 Theorem \ref{thm:Kclass} only scratches the surface of expectedly relevant features of $\mathscr{K}$-class operators, in view of the study of perturbations preserving Krylov solvability (see questions of type II in Sect.~\ref{sec:questions}).

 That the issue is non-trivial, however, is demonstrated by important difficulties that one soon encounters when trying to extend the scope of Theorem \ref{thm:Kclass}. Let us discuss here one point in particular: in the same spirit of questions of type II, it is natural to inquire whether $\mathscr{K}$-class operators allow to establish Krylov solvability in the limit when the perturbation is removed.

 To begin with, if a sequence $(A_n)_{n\in\mathbb{N}}$ of $\mathscr{K}$-class operators on $\cH$ converges in operator norm, the limit $A$ fails in general to be of $\mathscr{K}$-class. Indeed:
 
  \begin{lemma}\label{lem:Kclassnotclosed}
  The $\mathscr{K}$-class is not closed in $\mathcal{B}(\cH)$.
 \end{lemma}

 \begin{proof}
  It suffices to consider a positive, compact operator $A:\cH\to\cH$, with zero in its spectrum, and its perturbations $A_n:=A+n^{-1}\mathbbm{1}$, $n\in\mathbb{N}$. Then each $A_n$ is of $\mathscr{K}$-class and $\|A_n-A\|_{\mathrm{op}}\to 0$ as $n\to\infty$, but by construction $A$ is not of $\mathscr{K}$-class.
 \end{proof}

 One might be misled to believe that the general mechanism for such failure is the appearance of zero in the spectrum of the limit operator $A$, and that therefore a uniform separation of $\sigma(A_n)$ from zero as $n\to\infty$ would produce a limit $A$ still in the $\mathscr{K}$-class. To show that this is not the case either, let us work out the following example.
 
 \begin{example}\label{ex:lid-convergence}
  Let $A$ and $A_n$, $n\in\mathbb{N}$, be the operators on the Hilbert space $L^2[0,1]$ defined by
  \[
   \begin{split}
    (Af)(x)\;&:=\;e^{2\pi\ii\,x}f(x)\,, \\
    (A_n f)(x)\;&:=\;
    \begin{cases}
     e^{2\pi \ii\,x}f(x)\,, & \textrm{ if }x\in(\frac{1}{2\pi n},1]\,, \\
     (1+\frac{1}{n})\, e^{2\pi \ii\,x}f(x)\,, & \textrm{ if }x\in[0,\frac{1}{2\pi n}]\,,
    \end{cases}
   \end{split}
  \]
 for $f\in L^2[0,1]$ and a.e.~$x\in [0,1]$. Clearly, 
 %$\|A\|_{\mathrm{op}}=1$, $\|A_n\|_{\mathrm{op}}=1+\frac{1}{n}$.
 \[
  \|A\|_{\mathrm{op}}\,=\,1\,,\qquad \|A_n\|_{\mathrm{op}}\,=\,1+\frac{1}{n}\,,\qquad \|A-A_n\|_{\mathrm{op}}\,=\,\frac{1}{n}\, \|A\|_{\mathrm{op}}\;\xrightarrow{n\to\infty}\;0\,.
 \]
 Moreover, each $A_n$ is a $\mathscr{K}$-class operator: its spectrum $\sigma(A_n)$ covers the unit circle except for a `lid' arc, corresponding to the angle $x\in[0,\frac{1}{n}]$, which lies on the circle of larger radius $1+\frac{1}{n}$, therefore it is possible to include $\sigma(A_n)$ into a suitable bounded open $\mathcal{W}$ separated from zero and with connected complement in $\mathbb{C}$. In the limit $n\to\infty$ the `lid' closes the unit circle: $\sigma(A)$ is indeed the whole unit circle. Thus, even if the limit operator $A$ satisfies $0\notin\sigma(A)$, $A$ fails to belong to the $\mathscr{K}$-class.
 In addition, not only the $\mathscr{K}$-class condition is lost in the limit, but so too the Krylov solvability. Indeed, by means of the Hilbert space isomorphism
 \[
  L^2[0,1]\,\xrightarrow{\;\cong\;}\,\ell^2(\mathbbm{Z})\,,\qquad e^{2\pi\ii k x}\,\longmapsto e_k
 \]
 (namely with respect to the orthonormal bases $(e^{2\pi\ii k x})_{k \in \mathbbm{Z}}$ and $(e_k)_{k\in\mathbb{Z}}$), $A$ is unitarily equivalent to the right-shift operator on $\ell^2(\mathbbm{Z})$: thus, any choice $g \cong e_k$ for some $k\in\mathbb{Z}$ produces a non-Krylov solvable inverse problem $Af=g$.
 \end{example}

 In conclusion, the $\mathscr{K}$-class proves to be an informative sub-class of operators that is very robust and preserves Krylov solvability under perturbations of an unperturbed $\mathscr{K}$-class inverse problem (Theorem \ref{thm:Kclass}), but on the contrary is very fragile when from a sequence of approximating inverse problems of $\mathscr{K}$-class one wants to extract information on the Krylov solvability of the limit problem (Lemma \ref{lem:Kclassnotclosed}, Example \ref{ex:lid-convergence}).

 \section{Weak gap metric for weakly closed parts of the unit ball}\label{sec:weakgapmetric}
 
%  [Intro and motivation, depending on the findings of Section \ref{sec:gain-loss}.]
%  
%  \bigskip
 
 Let us introduce now a convenient indicator of vicinity of closed subspaces of a given Hilbert space, which turns out to possess convenient properties when comparing (closures of) Krylov subspaces, and to provide a rigorous language to express and control limits of the form $\overline{\mathcal{K}(A_n,g_n)}\to\overline{\mathcal{K}(A,g)}$. Even though such an indicator is not optimal, in that it lacks other desired properties that would make it fully informative, we discuss it in depth here as a first attempt towards an efficient measurement of vicinity and convergence of Krylov subspaces under perturbations.

 One natural motivation is provided by the \emph{failure} of describing the intuitive convergence $\mathcal{K}_N(A,g)\xrightarrow{N\to\infty}\overline{\mathcal{K}(A,g)}$, where 
 \begin{equation}\label{eq:defKrylov-N}
  \mathcal{K}_N(A,g) \;:=\;  \mathrm{span} \{g,Ag,\dots,A^{N-1}g\}\,,\qquad N\in\mathbb{N}
\end{equation}
 is the $N$-th order Krylov subspace, by means of the ordinary `\emph{gap metric}' between closed subspaces of the underlying Hilbert space.

 Let us recall (see, e.g., \cite[Chapt.~4, \S 2]{Kato-perturbation}) that given a Hilbert space $\cH$ and two closed subspaces $U,V\subset\cH$, the `\emph{gap}' and the `\emph{gap distance}' between them are, respectively, the quantities
 \begin{equation}
  \widehat{\delta}(U,V)\;:=\;\max\{\delta(U,V),\delta(V,U)\}
 \end{equation}
 and 
 \begin{equation}\label{eq:dhat}
  \widehat{d}(U,V)\;:=\;\max\{d(U,V),d(V,U)\}\,,
 \end{equation}
 where
 \begin{equation}\label{eq:dhat2}
  \begin{split}
    \delta(U,V)\;&:=\;\sup_{\substack{ u\in U \\ \|u\|=1 }}\inf_{\substack{ v\in V}}\|u-v\|\,, \\
    d(U,V)\;&:=\;\sup_{\substack{ u\in U \\ \|u\|=1 }}\inf_{\substack{ v\in V \\ \|v\|=1 }}\|u-v\|\,,  
  \end{split}
 \end{equation}
 and with the tacit definitions $\delta(\{0\},V):=0$, $d(\{0\},V):=0$, $d(U,\{0\}):=2$ for $U\neq\{0\}$, when one of the two entries is the empty set.
 The short-hands $B_\cH$ for the closed unit ball of $\cH$, $S_\cH$ for the closed unit sphere, $B_U:=U\cap B_\cH$, $S_U:=U\cap S_\cH$, and $\mathrm{dist}(x,C)$ for the norm distance of a point $x\in\cH$ from the closed subset $C\subset\cH$ will be used throughout. Thus,
 \begin{equation}
  \delta(U,V)\;=\;\sup_{u\in S_U}\mathrm{dist}(u,V)\,,\qquad d(U,V)\;=\;\sup_{u\in S_U}\mathrm{dist}(u,S_V)\,.
 \end{equation}
 As a matter of fact, on the set of all closed subspaces of $\cH$ both $\widehat{\delta}$ and $\widehat{d}$ are two equivalent metrics, with 
  \begin{equation}\label{eq:delta-d}
  \widehat{\delta}(U,V)\;\leqslant\; \widehat{d}(U,V)\;\leqslant\;2\widehat{\delta}(U,V)\,,
 \end{equation}
 and the resulting metric space is complete \cite{Gohberg-Markus-1959}.

 The construction that we recalled here is for the Hilbert space setting and was introduced first in 
 %by Kre\u{\i}n, and Krasnosel$'$ski\u{\i} 
 \cite{Krein-Krasnoselskii-1947} as `opening' between (closed) subspaces (i.e., the operator norm distance between their orthogonal projections). It also applies to the more general case when $\cH$ is a Banach space, a generalisation originally discussed in 
 %by Kre\u{\i}n, Krasnosel$'$ski\u{\i}, and Mil$'$man 
 in \cite{Krein-Krasnoselskii-Milman-1948}, and 
 %by Akhiezer and Glazman 
 \cite[\S 34]{Akhiezer-Glazman-1961-1993} (except that in the non-Hilbert case the gap $\widehat{\delta}$ is not a metric, even though it still satisfies \eqref{eq:delta-d} and hence induces the same topology as the metric $\widehat{d}$). Let us also recall that by linearity the closedness of the above subspaces $U$ and $V$ can be equivalently formulated in the norm or in the weak topology of $\cH$.

 Now, given $A\in\mathcal{B}(\cH)$ and $g\in\cH$, for the closed subspaces $\mathcal{K}:=\overline{\mathcal{K}(A,g)}$ and $\mathcal{K}_N:=\mathcal{K}_N(A,g)$, $N\in\mathbb{N}$, of $\cH$ one obviously has $\mathcal{K}_N\subset \mathcal{K}$ and hence $\delta(\mathcal{K}_N,\mathcal{K})=0$; on the other hand, if $\dim\mathcal{K}=\infty$, one can find for every $N$ a vector $u\in S_{\mathcal{K}}$ such that $u\perp \mathcal{K}_N$, thus with $\mathrm{dist}(u,\mathcal{K}_N)=1$, and hence $\delta(\mathcal{K},\mathcal{K}_N)\geqslant 1$. This shows that $\widehat{d}(\mathcal{K}_N,\mathcal{K})\geqslant \widehat{\delta}(\mathcal{K}_N,\mathcal{K})\geqslant 1$, therefore the sequence $(\mathcal{K}_N)_{N\in\mathbb{N}}$ fails to converge to $\mathcal{K}$ in the $\widehat{d}$-metric. In this respect, the $\widehat{d}$-metric is certainly not a convenient tool to monitor the vicinity of Krylov subspaces, for it cannot accommodate the most intuitive convergence $\mathcal{K}_N\to\mathcal{K}$.

 With this observation in mind, it is natural to weaken the ordinary gap distance $\widehat{d}$-metric so as to encompass a larger class of limits. To do so, we exploit the fact (see, e.g., \cite[Theorem 3.29]{Brezis-FA-Sob-PDE}) that in any separable Hilbert space $\cH$ 
 %(as well as, more generally, in any Banach space with separable dual) 
 the norm-closed unit ball $B_\cH$ is metrisable in the Hilbert space weak topology. More precisely, there exists a norm $\|\cdot\|_w$ on $\cH$ (and hence a metric $\varrho_w(x,y):=\|x-y\|_w$) such that $\|x\|_w\leqslant\|x\|$ and whose metric topology restricted to $B_\cH$ is precisely the Hilbert space weak topology. For concreteness one may define
 \[
  \|x\|_w\;:=\;\sum_{n=1}^\infty \frac{1}{\:2^n}|\langle \xi_n,x\rangle|
 \]
 for a dense countable collection $(\xi_n)_{n\in\mathbb{N}}$ in $B_\cH$ which identifies the norm $\|\cdot\|_w$. On the other hand, since a Hilbert space is reflexive, $B_\cH$ is compact in the weak topology (see, e.g., \cite[Theorem 3.16]{Brezis-FA-Sob-PDE}), and hence in the $\varrho_w$-metric. Being $(B_\cH,\varrho_w)$ a metric space, its compactness is equivalent to the property of being simultaneously complete and totally bounded (see, e.g., \cite[Theorem 45.1]{Munkres-Topology}). In conclusion, the metric space $(B_\cH,\varrho_w)$ is compact and complete, and its metric topology is the Hilbert space weak topology (restricted to $B_\cH$). In fact, the construction that follows, including Theorem \ref{thm:weak-gap-properties-and-completeness-on-ball} below, is applicable to the more general case where $\cH$ is a reflexive Banach space with separable dual: indeed, the same properties above for $(B_\cH,\varrho_w)$ hold.

 In $(B_\cH,\varrho_w)$ we denote the relative \emph{weakly open} balls (namely the $\varrho_w$-\emph{open} balls of $\cH$ intersected with $B_\cH$) as 
 \begin{equation}
  \mathfrak{B}_w(x_0,\varepsilon)\;:=\;\{x\in B_\cH\,|\,\|x-x_0\|_w<\varepsilon\}
 \end{equation}
 for given $x_0\in B_\cH$ and $\varepsilon>0$. 
 %The weak topology of $B_\cH$ being a metric topology, it is in particular second countable.
 Observe that any such open ball $\mathfrak{B}_w(x_0,\varepsilon)$ always contains points of the unit sphere (not all, if $\varepsilon$ is small enough); thus, at fixed $x_0\in B_\cH$, and along a sequence of radii $\varepsilon_n\downarrow 0$, one can select a sequence $(y_n)_{n\in\mathbb{N}}$ with $\|y_n\|=1$ and $\|y_n-x_0\|_w<\varepsilon_n$, whence the conclusion $y_n\xrightarrow{\varrho_w}x_0$, which reproduces, in the metric space language, the topological statement that $y_n\rightharpoonup x_0$, i.e., that the unit ball is the weak closure of the unit sphere.

 Based on the weak (and metric) topology $B_\cH$ it is natural to weaken the gap distance $\widehat{d}$ considered before, as we shall do in a moment, except that dealing now with weak limits instead of norm limits one has to expect possible ``discontinuous jumps'', say, in the form of sudden expansions or contractions of the limit object as compared to its approximants (in the same spirit of taking the closure of the unit sphere $S_\cH$: the norm-closure gives again $S_\cH$, the weak closure gives the whole $B_\cH$). For this reason we set up the new notion of weak gap-metric in the more general class
 \begin{equation}
  \mathcal{C}_w(\cH)\;:=\;\{\textrm{ non-empty and weakly closed subsets of $B_\cH$ }\}\,,
 \end{equation}
 instead of the subclass of unit balls of closed subspaces of $\cH$.

 For $U,V\in\mathcal{C}_w(\cH)$ let us then set 
 \begin{equation}\label{eq:def-dw-in-ball}
  \begin{split}
   d_w(U,V)\;&:=\;\sup_{u\in U}\inf_{v\in V}\|u-v\|_w\,,  \\ %\sup_{\substack{ u\in U \\ \|u\|\leqslant 1 }}\inf_{\substack{ v\in V \\ \|v\|=1 }}\|u-v\|_w %\;=\;\sup_{u\in B_U}\mathrm{dist}_w(u,B_V)
   \widehat{d}_w(U,V)\;&:=\;\max\{ d_w(U,V) , d_w(V,U) \}\,.
  \end{split}
 \end{equation}

 %In a moment we shall 
 We shall now establish the fundamental properties of the map $\widehat{d}_w$. They are summarised as follows.

 \begin{theorem}\label{thm:weak-gap-properties-and-completeness-on-ball}
  Let $\cH$ be a separable Hilbert space.
  \begin{itemize}
   \item[(i)] $\widehat{d}_w$ is a metric on $\mathcal{C}_w(\cH)$.
   \item[(ii)] The metric space $(\mathcal{C}_w(\cH),\widehat{d}_w)$ is complete.
   \item[(iii)] If $\widehat{d}_w(U_n,U)\xrightarrow{n\to\infty} 0$ for an element $U$ and a sequence $(U_n)_{n\in\mathbb{N}}$ in $\mathcal{C}_w(\cH)$, then
   \begin{equation}\label{eq:limit-U}
    U\;=\;\{u\in B_{\cH}\,|\,u_n\rightharpoonup u\textrm{ for a sequence }(u_n)_{n\in\mathbb{N}} \textrm{ with } u_n\in U_n\}\,.
   \end{equation}
   \item[(iv)] The metric space $(\mathcal{C}_w(\cH),\widehat{d}_w)$ is compact.
   \item[(v)] If $\widehat{d}_w(U_n,U)\xrightarrow{n\to\infty} 0$ for an element $U$ and a sequence $(U_n)_{n\in\mathbb{N}}$ in $\mathcal{C}_w(\cH)$, then $\widehat{d}_w(f(U_n),f(U))\xrightarrow{n\to\infty} 0$ for any weakly closed and weakly continuous map $f:\cH\to\cH$ such that $f(B_\cH)\subset B_\cH$.
  \end{itemize}
 \end{theorem}

  We shall also write $U_n\xrightarrow{\widehat{d}_w}U$ as an alternative to $\widehat{d}_w(U_n,U)\to 0$.

  \begin{remark}
 The completeness and the compactness result of Theorem \ref{thm:weak-gap-properties-and-completeness-on-ball} are in a sense folk knowledge in the context of the Hausdorff distance. In fact, the gap distance $\widehat{d}(U,V)$ introduced in \eqref{eq:dhat}-\eqref{eq:dhat2} is, apart from zero-sets, the Hausdorff distance between $U$ and $V$ as subsets of the metric (normed) space $(\cH,\|\cdot\|)$, and our modified weak gap distance $\widehat{d}_w(U,V)$ defined in \eqref{eq:def-dw-in-ball} between elements of $\mathcal{C}_w(\cH)$ is the Hausdorff distance between sets in the metric space $(B_H,\varrho_w)$. The completeness and the compactness of $(B_H,\varrho_w)$ then lift, separately, to the completeness and compactness of $(\mathcal{C}_w(\cH),\widehat{d}_w)$ -- they are actually equivalent (see, e.g., \cite[Theorem 5.38]{Tuzhilin-2020}. We chose to present here both results and their proofs in detail for three important reasons. First, we wanted to make the discussion self-consistent (also in view of the rather miscellaneous literature we could track down, our proof of completeness, in particular, following an independent route than the general discussion \cite{Henrikson-1999,Gupta-Mukherjee-2019,Tuzhilin-2020}). Second, we intended to expose reasonings, tailored on the weak topology setting, which we shall use repeatedly in the proof of the various statements of the following Sections. Third, having the proof of completeness of $(\mathcal{C}_w(\cH),\widehat{d}_w)$ fully laid down is of further help in understanding the \emph{failure} of completeness of the metric space $(\mathscr{S}(\cH),\widehat{d}_w)$ that we will consider in the next 
Section, for applications to Krylov subspaces. 
\end{remark}

  There are further technical properties of $d_w$ and $\widehat{d}_w$ that are worth being singled out. Let us collect them in the following lemmas.
  
  \begin{lemma}\label{lem:inclusion-and-triangular}
   Let $U,V,Z\in \mathcal{C}_w(\cH)$ for some separable Hilbert space $\cH$. Then 
   \begin{eqnarray}
    d_w(U,V)\,=\,0 &\Leftrightarrow& U\subset V \,, \label{eq:dwinclusion} \\
    \widehat{d}_w (U,V)\,=\,0 &\Leftrightarrow& U= V \,, \label{eq:dhatzero}\\
    d_w(U,Z)&\leqslant& d_w(U,V)+ d_w(V,Z)\,, \label{eq:d-triangular} \\
    \widehat{d}_w(U,Z)&\leqslant&\widehat{d}_w(U,V)+\widehat{d}_w(V,Z)\,. \label{eq:dhat-triangular}
   \end{eqnarray}
  \end{lemma}

 \begin{lemma}\label{lem:weaklimits-closed}
  Given a separable Hilbert space $\cH$ and a collection $(U_n)_{n\in\mathbb{N}}$ in $\mathcal{C}_w(\cH)$, the set 
  \begin{equation}
   \mathcal{U}\;:=\;\{x\in B_\cH\,|\,u_n\rightharpoonup x\textrm{ for a sequence }(u_n)_{n\in\mathbb{N}} \textrm{ with } u_n\in U_n\}
  \end{equation}
  is closed in the weak topology of $\cH$.
 \end{lemma}

 Given $U\in \mathcal{C}_w(\cH)$ we define its `\emph{weakly open $\varepsilon$-expansion}' in $B_\cH$ as 
 \begin{equation}\label{eq:epsexp}
  U(\varepsilon)\;:=\;\bigcup_{u\in U}\mathfrak{B}_w(u,\varepsilon)\,.
 \end{equation}
 Observe that $U(\varepsilon)$ is a weakly open subset of $B_\cH$.

 \begin{lemma}\label{lem:38}
  Let $U,V\in \mathcal{C}_w(\cH)$ for some separable Hilbert space $\cH$ and let $\varepsilon>0$. Then:
  \begin{itemize}
   \item[(i)] $d_w(U,V)<\varepsilon$ $\Rightarrow$ $V\cap \mathfrak{B}_w(u,\varepsilon)\neq\emptyset$ $\forall u\in U$; 
   \item[(ii)] $d_w(U,V)<\varepsilon$ $\Leftrightarrow$ $U\subset V(\varepsilon)$;
   \item[(iii)] $\left.\begin{array}{r} U\subset V(\varepsilon) \\ U\cap \mathfrak{B}_w(v,\varepsilon)\neq\emptyset\;\forall v\in V\end{array}\!\right\}$ $\Rightarrow$ $ \widehat{d}_w (V,U)<\varepsilon$.
  \end{itemize}
 \end{lemma}

 The remaining part of this Section is devoted to proving the above statements.
 
 \begin{proof}[Proof of Lemma \ref{lem:inclusion-and-triangular}]
  For \eqref{eq:dwinclusion}, the inclusion $U\subset V$ implies $\inf_{v\in V}\|u-v\|_w=0$ for every $u\in U$, whence $d_w(U,V)=0$; conversely, if $0=d_w(U,V)=\sup_{u\in U}\inf_{v\in V}\|u-v\|_w$, then $\inf_{v\in V}\|u-v\|_w=0$ for every $u\in U$, whence the fact, by weak closedness of $V$, 
  %as a closed subset of the compact metric space $(B_\cH,\varrho_w)$, 
  that any such $u$ belongs also to $V$. As for the property \eqref{eq:dhatzero}, it follows from \eqref{eq:dwinclusion} exploiting separately both inclusions $U\subset V$ and $U\supset V$. Last, let us prove the triangular inequalities \eqref{eq:d-triangular}-\eqref{eq:dhat-triangular}. Let $u_0\in U$: then, owing to the weak compactness of $V$ (as a closed subset of the compact metric space $(B_\cH,\varrho_w)$), $\inf_{v\in V}\|u_0-v\|_w=\|u_0-v_0\|_w$ for some $v_0\in V$, whence  
  \[
   \|u_0-v_0\|_w\;=\;\inf_{v\in V}\|u_0-v\|_w\;\leqslant\;  \sup_{u\in U}\inf_{v\in V}\|u-v\|_w \;=\;    d_w(U,V)\;\leqslant\;\widehat{d}_w(U,V)\,.
  \]
  As a consequence,
  \[
   \begin{split}
    \inf_{z\in Z}\|u_0-z\|_w\;&\leqslant\;\|u_0-v_0\|_w+ \inf_{z\in Z}\|v_0-z\|_w \\
    &\leqslant\; d(U,V)+d(U,Z) \\
    &\leqslant\;\widehat{d}_w(U,V)+\widehat{d}_w(V,Z)\,,
   \end{split}
  \]
  having used the triangular inequality of the $\|\cdot\|_w$-norm in the first inequality. By the arbitrariness of $u_0\in U$, thus taking the supremum over all such $u_0$'s,
  \[
   \begin{split}
       d_w(U,Z)\;&\leqslant\;d_w(U,V)+d_w(V,Z)\,, \\
       d_w(U,Z)\;&\leqslant\;\widehat{d}_w(U,V)+\widehat{d}_w(V,Z)\,.
   \end{split}
  \]
  With the first inequality above we proved \eqref{eq:d-triangular}. Next, let us combine the second inequality above with the corresponding bound for $d_w(V,U)$, which is established in a similar manner: let now $z_0\in Z$, and again by weak compactness there exists $v_0\in V$ with $\|v_0-z_0\|_w=\inf_{v\in V}\|v-z_0\|_w\leqslant d_w(Z,V)\leqslant\widehat{d}_w(V,Z)$, and also $\inf_{u\in U}\|u-v_0\|_w\leqslant d(V,U)=\widehat{d}(U,V)$, whence
  \[
   \inf_{u\in U}\|u-z_0\|_w\;\leqslant\;\inf_{u\in U}\|u-v_0\|_w+ \|v_0-z_0\|_w\;\leqslant\;\widehat{d}_w(U,V)+\widehat{d}_w(V,Z)\,.
  \]
   Taking the supremum over all $z_0\in Z$ yields
     \[
   d_w(Z,U)\;\leqslant\;\widehat{d}_w(U,V)+\widehat{d}_w(V,Z)\,.
  \]
   Combining the above estimates for $d_w(U,Z)$ and $d_w(Z,U)$ yields the conclusion.
 \end{proof}

  \begin{proof}[Proof of Theorem \ref{thm:weak-gap-properties-and-completeness-on-ball}(i)]
   It follows directly from \eqref{eq:dhatzero}-\eqref{eq:dhat-triangular} of Lemma \ref{lem:inclusion-and-triangular}.
  \end{proof}

  \begin{proof}[Proof of Lemma \ref{lem:weaklimits-closed}]
   Let $x\in\overline{\mathcal{U}}^{w}$, the closure of $\mathcal{U}$ in the weak topology, and let us construct a sequence $(u_n)_{n\in\mathbb{N}}$ with $u_n\in U_n$ and $u_n\rightharpoonup x$, thereby showing that $x\in\mathcal{U}$. 
   
   By assumption $\exists\, x_2\in\mathcal{U}$ with $\|x-x_2\|_w\leqslant\frac{1}{4}$ and $\|x_2-u_n^{(2)}\|_w\xrightarrow{n\to\infty} 0$ for a sequence $( u_n^{(2)})_{n\in\mathbb{N}}$ with $u_n^{(2)}\in U_n$. In particular, there is $N_2\in\mathbb{N}$ with $\|x_2-u_n^{(2)}\|_w\leqslant\frac{1}{4}$ $\forall n\geqslant N_2$. For the integer $N_3$ with $N_3>N_2+1$ to be fixed in a moment, set 
   \[
    u_n\;:=\; u_n^{(2)}\,,\qquad n\in\{N_2,\dots,N_3-1\}\,.
   \]
  By construction, $\|x-u_n\|_w\leqslant\|x-x_n\|_w+\|x_n-u_n\|_w\leqslant\frac{1}{2}$ $\forall n\in\{N_2,\dots,N_3-1\}$.

   Next, for each integer $k\geqslant 3$ one identifies recursively a sequence $(N_k)_{k=3}^\infty$ in $\mathbb{N}$ with $N_k>N_{k-1}+1$, and vectors $u_n\in U_n$ for $n\in\{N_k,\dots,N_{k+1}-1\}$ as follows.
    By assumption $\exists\, x_k\in\mathcal{U}$ with $\|x-x_k\|_w\leqslant\frac{1}{2k}$ and $\|x_k-u_n^{(k)}\|_w\xrightarrow{n\to\infty} 0$ for a sequence $( u_n^{(k)})_{n\in\mathbb{N}}$ with $u_n^{(k)}\in U_n$. In particular, it is always possible to find $N_k\in\mathbb{N}$ with $N_k>N_{k-1}+1$ such that   
    $\|x_k-u_n^{(k)}\|_w\leqslant\frac{1}{2k}$ $\forall n\geqslant N_k$. For the integer $N_{k+1}>N_k+1$ set 
   \[
    u_n\;:=\; u_n^{(k)}\,,\qquad n\in\{N_k,\dots,N_{k+1}-1\}\,.
   \]
  By construction, $\|x-u_n\|_w\leqslant\|x-x_n\|_w+\|x_n-u_n\|_w\leqslant\frac{1}{k}$ $\forall n\in\{N_k,\dots,N_{k+1}-1\}$.

   This yields a sequence $(u_n)_{n\in\mathbb{N}}$ (having added, if needed, finitely many irrelevant vectors $u_1,\dots,u_{N_2-1}$) with $u_n\in U_n$ and such that, for any integer $k\geqslant 2$,
   \[
    \|x-u_n\|_w\leqslant\frac{1}{k}\quad\forall\,n\geqslant N_k\,.
   \]
   Hence $\|x-u_n\|_w\xrightarrow{n\to\infty} 0$, thus $x\in\mathcal{U}$.   
  \end{proof}

   \begin{proof}[Proof of Theorem \ref{thm:weak-gap-properties-and-completeness-on-ball}(ii) and (iii)]
    One needs to show that given $(U_n)_{n\in\mathbb{N}}$, Cauchy sequence in $\mathcal{C}_w(\cH)$, there exists $U\in\mathcal{C}_w(\cH)$ with $\widehat{d}_w(U_n,U)\xrightarrow{n\to\infty}0$, and that $U$ has precisely the form \eqref{eq:limit-U}. Moreover, as $(\mathcal{C}_w(\cH),\widehat{d}_w)$ is a metric space, it suffices to establish the above statement for one subsequence of $(U_n)_{n\in\mathbb{N}}$.

    By the Cauchy property, $\widehat{d}_w(U_n,U_m)\xrightarrow{n,m\to\infty}0$. Up to extracting a subsequence, henceforth denoted again with $(U_n)_{n\in\mathbb{N}}$, one can further assume that 
    \[
     \widehat{d}_w(U_n,U_m)\;\leqslant\;\frac{1}{\:2^n}\quad \forall m\geqslant n\,.
    \]
    We shall establish the $\widehat{d}_w$-convergence of such (sub-)sequence.

    First of all, fixing any $n\in\mathbb{N}$ and any $u_{n}\in U_{n}$, we construct a sequence 
    %$(u_k)_{k\in\mathbb{N}}$ 
    with 
    \begin{itemize}
     \item an (irrelevant) choice of vectors $u_1,\dots,u_{n-1}$ in the first $n-1$ positions, such that $u_1\in U_1,\dots,u_{n-1}\in U_{n-1}$,
     \item precisely the considered vector $u_n$ in position $n$,
     \item and an infinite collection $u_{n+1},u_{n+2},u_{n+3},\dots$ determined recursively so that, given $u_k\in U_k$ ($k\geqslant n$), the next $u_{k+1}$ is that element of $U_{k+1}$ satisfying $\inf_{v\in U_{k+1}}\|u_k-v\|_w=\|u_k-u_{k+1}\|_w$ -- a choice that is always possible, owing to the weak compactness of $U_{k+1}$ as a closed subset of the compact metric space $(B_\cH,\varrho_w)$.
    \end{itemize}
    Let us refer to such $(u_k)_{k\in\mathbb{N}}$ as the sequence `originating from the given $u_n$' (tacitly understanding that it is one representative of infinitely many sequences with the same property, owing to the irrelevant choice of the first $n-1$ vectors). When the originating vector need be indicated, we shall write $\big(u^{(u_n)}_k\big)_{k\in\mathbb{N}}$\,: thus, $u^{(u_n)}_n\equiv u_n$.

    By construction, for any $k\geqslant n$, 
    %(only the tail of the sequence counts, the first $n-1$ elements are irrelevant),
    \[
     \begin{split}
      \|u_k-u_{k+1}\|_w\;&=\;\inf_{v\in U_{k+1}}\|u_k-v\|_w\;\leqslant\;\sup_{z\in U_k}\inf_{v\in U_{k+1}}\|z-v\|_w \\
      &=\;d_w(U_{k},U_{k+1})\;\leqslant\;\widehat{d}_w(U_{k},U_{k+1})\;\leqslant\;\frac{1}{\:2^k}\,,
     \end{split}
    \]
    whence, for any $m> n$,
    \[
     \|u_n-u_m\|_w\;\leqslant\;\sum_{k=n}^{m-1}\|u_k-u_{k+1}\|_w\;\leqslant\;\sum_{k=n}^{m-1}\frac{1}{\:2^k}\;\leqslant\;\frac{1}{\:2^{n-1}}\,.
    \]
    This implies that the sequence $(u_k)_{k\in\mathbb{N}}$ originating from the considered $u_n\in U_n$ is a Cauchy sequence in $(B_\cH,\varrho_w)$ and we denote its weak limit as $u_\infty^{(u_n)}\in B_{\cH}$. The same construction can be repeated for any $n\in\mathbb{N}$ and starting the sequence from any $u_n\in U_n$: the collection of all possible limit points is
    \[
     U_\infty\;:=\;\big\{u\in B_\cH\,|\,u=u_\infty^{(u_n)}\;\textrm{ for some $n\in\mathbb{N}$ and some `starting' $u_n\in U_n$}\big\}\,.
    \]
    %$U$ defined by  \eqref{eq:limit-U}.

    Compare now the set $U_\infty$ with the set 
    \[
     \widetilde{U}\;:=\;\{u\in B_\cH\,|\,u_n\rightharpoonup u\textrm{ for a sequence }(u_n)_{n\in\mathbb{N}} \textrm{ with } u_n\in U_n\}
    \]
    $\widetilde{U}$ is weakly closed (Lemma \ref{lem:weaklimits-closed}), and obviously $U_\infty\subset \widetilde{U}$. We claim that
    \[
     \widetilde{U}\;=\;\overline{U_\infty}^{\|\,\|_w}\qquad\textrm{(the weak closure of $U_\infty$)}\,.
    \]
    For arbitrary $u\in \widetilde{U}$ and $\varepsilon>0$ there is $n_\varepsilon\in\mathbb{N}$ and $u_{n_\varepsilon}\in U_{n_\varepsilon}$ with $\|u-u_{n_\varepsilon}\|_w\leqslant\varepsilon$. Non-restrictively, $n_\varepsilon\to\infty$ as $\varepsilon\downarrow 0$.  For a sequence $\big(u^{(u_{n_\varepsilon})}_k\big)_{k\in\mathbb{N}}$ originating from $u_{n_\varepsilon}$ and for its weak limit $u_\infty^{(u_{n_\varepsilon})}\in U_\infty$, there is $k_\varepsilon\in\mathbb{N}$ with $k_\varepsilon> n_\varepsilon$ satisfying both $\|u_{n_\varepsilon}-u^{(u_{n_\varepsilon})}_{k_\varepsilon}\|_w\leqslant 2^{-(n_\varepsilon-1)}$ (because of the above property of the sequences originating from one element) and $\|u^{(u_{n_\varepsilon})}_{k_\varepsilon}-u_\infty^{(u_{n_\varepsilon})}\|_w\leqslant\varepsilon$ (because of the convergence $u^{(u_{n_\varepsilon})}_k\rightharpoonup u_\infty^{(u_{n_\varepsilon})}$). Thus,
    \[
     \begin{split}
           \|u-u_\infty^{(u_{n_\varepsilon})}\|_w\;&\leqslant\;\|u-u_{n_\varepsilon}\|_w+\|u_{n_\varepsilon}-u^{(u_{n_\varepsilon})}_{k_\varepsilon}\|_w+\|u^{(u_{n_\varepsilon})}_{k_\varepsilon}-u_\infty^{(u_{n_\varepsilon})}\|_w \\
           &\leqslant\;2^{-(n_\varepsilon-1)}+2\varepsilon\,.
     \end{split}
    \]
    Taking $\varepsilon\downarrow 0$ shows that $u$ indeed belongs to the weak closure of $U_\infty$.

    %%%%%%%%%%%%%%%%%%%%%%%%%%%%%%%%%
    %%%%%%%%%%%%%%%%%%%%%%%%%%%%%%%%% FURTHER EXPLANATION TO POSSIBLY ADD
    %%%%%%%%%%%%%%%%%%%%%%%%%%%%%%%%%
%     Consider the possibility that $\widehat{d}_w(U_{n},\widetilde{U})\xrightarrow{n\to\infty}0$. This would mean that among all Cauchy sequences in $(\mathcal{C}_w(\cH),\widehat{d}_w)$, the subclass of `squeezed' Cauchy sequences $(U_n)_{n\in\mathbb{N}}$ characterised by the above `squeezing' condition  
%     \[
%      \widehat{d}_w(U_n,U_m)\;\leqslant\;\frac{1}{\:2^n}\quad \forall m\geqslant n
%     \]
%     are all convergent, each to the corresponding $\widetilde{U}$, the set consisting of all weak limit points of vector sequences, the $k$-th element of which belongs to the $k$-th set of the squeezed sequence $(U_n)_{n\in\mathbb{N}}$.
%     As argued at the beginning, any Cauchy sequence in $(\mathcal{C}_w(\cH),\widehat{d}_w)$ has a squeezed subsequence. Therefore any Cauchy sequence converges to the set $U$, defined by  \eqref{eq:limit-U}, consisting of all weak limit points of vector sequences, the $k$-th element of which belongs to the $k$-th set of the considered sequence.
%      This explains that to conclude the proof we are only left with demonstrating that $\widehat{d}_w(U_{n},\widetilde{U})\xrightarrow{n\to\infty}0$ for a squeezed Cauchy sequence $(U_n)_{n\in\mathbb{N}}$.
    %%%%%%%%%%%%%%%%%%%%%%%%%%%%%%%%%
    %%%%%%%%%%%%%%%%%%%%%%%%%%%%%%%%% END OF FURTHER EXPLANATION
    %%%%%%%%%%%%%%%%%%%%%%%%%%%%%%%%% 
     
   It remains to prove that $\widehat{d}_w(U_{n},\widetilde{U})\xrightarrow{n\to\infty}0$.
   Let us control $d_w(U_{n},\widetilde{U})$ first.
   Pick $n\in\mathbb{N}$ and $u_n\in U_n$. For the sequence $(u_k)_{k\in\mathbb{N}}$ originating from $u_n$ (thus, $u_k\rightharpoonup u_\infty^{(u_n)}$) and for arbitrary $\varepsilon>0$ there is $k_\varepsilon\in\mathbb{N}$ with $ k_\varepsilon>n$ such that $\|u_\infty^{(u_n)}-u_{k_\varepsilon}\|_w\leqslant\varepsilon$ and $\|u_{k_\varepsilon}-u_n\|_w\leqslant 2^{-(n-1)}$. Thus,
   \[
    \inf_{u\in \widetilde{U}}\|u-u_n\|_w\;\leqslant\;\|u_\infty^{(u_n)}-u_n\|_w\;\leqslant\;\|u_\infty^{(u_n)}-u_{k_\varepsilon}\|_w+\|u_{k_\varepsilon}-u_n\|_w\;\leqslant\;\varepsilon+\frac{1}{\:2^{n-1}}\,,
   \]
   whence also
   \[
    d_w(U_n,\widetilde{U})\;=\;\sup_{u_n\in U_n}\inf_{u\in \widetilde{U}}\|u-u_n\|_w\leqslant\;\varepsilon+\frac{1}{\:2^{n-1}}\,.
   \]
   This implies that $\lim\sup_n d_w(U_n,\widetilde{U})\leqslant\varepsilon$, and owing to the arbitrariness of $\varepsilon$, finally $d_w(U_n,\widetilde{U})\xrightarrow{n\to\infty}0$. The other limit $d_w(\widetilde{U},U_n)\xrightarrow{n\to\infty}0$ is established in much the same way, exploiting additionally the density of $U_\infty$ in $\widetilde{U}$. Pick arbitrary $n\in\mathbb{N}$, $\varepsilon>0$, and $u\in \widetilde{U}$. We already argued, for the proof of the identity $\widetilde{U}=\overline{U_\infty}^{\|\,\|_w}$, that there is $n_\varepsilon\in \mathbb{N}$ with $n_\varepsilon\to\infty$ as $\varepsilon\downarrow 0$, and there is $u_{n_\varepsilon}\in U_{n_\varepsilon}$, 
   %$k_\varepsilon\in\mathbb{N}$ with $k_\varepsilon>n_\varepsilon$, 
   such that
   \[
    \big\|u-u_\infty^{(u_{n_\varepsilon})}\big\|_w\;\leqslant\;2^{-(n_\varepsilon-1)}+2\varepsilon\,.
   \]
   Non-restrictively, $n_\varepsilon>n$. In turn, as $u^{(u_{n_\varepsilon})}_k\rightharpoonup u_\infty^{(u_{n_\varepsilon})}$, there is $k_\varepsilon\in\mathbb{N}$ with $k_\varepsilon\geqslant n_\varepsilon>n$ such that $\big\|u_\infty^{(u_{n_\varepsilon})}-u_{k_\varepsilon}^{(n_\varepsilon)}\big\|_w\leqslant\varepsilon$ and $\big\|u_{k_\varepsilon}^{(n_\varepsilon)}-u_{n}^{(n_\varepsilon)}\big\|\leqslant 2^{-(n-1)}$. Therefore, 
  \[
   \begin{split}
    \inf_{v\in U_{n}}&\|u-w\|_w\;\leqslant\; \big\|u-u_\infty^{(u_{n_\varepsilon})}\big\|_w+\inf_{v\in U_{n}}\big\|u_\infty^{(u_{n_\varepsilon})}-v\big\|_w \\
    &\leqslant\; \big\|u-u_\infty^{(u_{n_\varepsilon})}\big\|_w+\big\|u_\infty^{(u_{n_\varepsilon})}-u_{k_\varepsilon}^{(n_\varepsilon)}\big\|_w+\big\|u_{k_\varepsilon}^{(n_\varepsilon)}-u_{n}^{(n_\varepsilon)}\big\| \\
   &\leqslant\;\frac{1}{\;2^{n_\varepsilon-1}}+3\varepsilon+\frac{1}{\;2^{n-1}}\,,
   \end{split}
  \]
  whence also
   \[
    d_w(U,U_n)\;=\;\sup_{u\in U}\inf_{u_n\in U_n}\|u-u_n\|_w\leqslant\;\frac{1}{\;2^{n_\varepsilon-1}}+3\varepsilon+\frac{1}{\;2^{n-1}}\,.
   \]
   As above, the limit $\varepsilon\downarrow 0$ and the arbitrariness of $n$ imply  $d_w(U,U_n)\xrightarrow{n\to\infty}0$ and finally $\widehat{d}_w(U_n,U)\xrightarrow{n\to\infty}0$.
   \end{proof}

   \begin{proof}[Proof of Lemma \ref{lem:38}] (i) If, for contradiction, $V\cap \mathfrak{B}_w(u_0,\varepsilon)=\emptyset$ for some $u_0\in U$, then the weak metric distance ($\varrho_w$) of $u_0$ from $V$ is at least $\varepsilon$, meaning that 
 \[
  d_w(U,V)\;\geqslant\;\inf_{v\in V}\|u_0-v\|_w\;\geqslant\;\varepsilon\,.
 \]
 
 (ii) Assume that $d_w(U,V)<\varepsilon$. On account of the weak compactness of $V$ (as a closed subset of the compact metric space $(B_\cH,\varrho_w)$), for any $u\in U$ there is $v_u\in V$ with
 \[
  \|u-v_u\|_w\;=\inf_{v\in V}\|u-v\|_w\;\leqslant\;d_w(U,V)\;<\;\varepsilon\,,
 \]
 meaning that $u\in \mathfrak{B}_w(v_u,\varepsilon)$. Thus, $U\subset V(\varepsilon)$. Conversely, if $U\subset V(\varepsilon)$, then any $u\in U$ belongs to a ball $\mathfrak{B}_w(v_u,\varepsilon)$ for some $v_u\in V$, whence
 \[
  f(u)\;:=\;\inf_{v\in V}\|u-v\|_w\;\leqslant\;\|u-v_u\|_w\;<\;\varepsilon\,.
 \]
  The function $f:U\to\mathbb{R}$ is continuous on the weak compact set $U$, hence it attains its maximum at a point $u=u_0$ and 
  \[
   d_w(U,V)\;=\;\sup_{u\in U}f(u)\;=\;\inf_{v\in V}\|u_0-v\|_w\;<\;\varepsilon\,.
  \]

 (iii) As by assumption $U\subset V(\varepsilon)$, we know from (ii) that $d_w(U,V)<\varepsilon$. In addition, for any $v\in V$ it is assumed that $\mathfrak{B}_w(v,\varepsilon)$ is not disjoint from $U$, meaning that there is $u_v\in U$ with $\|u_v-v\|_w<\varepsilon$. Therefore,
 \[
  g(v)\;:=\;\inf_{u\in U}\|u-v\|_w\;<\;\varepsilon,
 \]
  and from the continuity of $g:V\to\mathbb{R}$ on the weak compact $V$,
  \[
   d_w(V,U)\;=\;\sup_{v\in V}g(v)\;=\;g(v_0)\;<\;\varepsilon\,,
  \]
  where $v_0$ is some point of maximum for $g$. In conclusion, $\widehat{d}_w (V,U)<\varepsilon$.
 \end{proof}

 \begin{proof}[Proof of Theorem \ref{thm:weak-gap-properties-and-completeness-on-ball}(iv)]
  As the metric space $(\mathcal{C}_w(\cH),\widehat{d}_w)$ is complete, compactness follows if one proves that for any $\varepsilon>0$ the set $\mathcal{C}_w(\cH)$ can be covered by finitely many $\widehat{d}_w$-open balls of radius $\varepsilon$ (total boundedness and completeness indeed imply compactness for a metric space).	
  
  To this aim, let us observe first that, owing to the compactness of $(B_\cH,\varrho_w)$, for any $\varepsilon>0$ we may cover it with finitely many open balls $\mathfrak{B}_w(x_1,\varepsilon),\dots,\mathfrak{B}_w(x_{M},\varepsilon)$ for some $x_1,\dots x_M\in B_\cH$ and $M\in\mathbb{N}$ all depending on $\varepsilon$. Each $\mathfrak{B}_w(x_n,\varepsilon)$ is the $\varepsilon$-expansion of the weakly closed set $\{x_n\}$, hence
  \[
   Z(\varepsilon)\;=\;\bigcup_{x\in Z}\mathfrak{B}_w(x,\varepsilon) \qquad\forall\, Z\subset\mathcal{Z}_M\,:=\,\{x_1,\dots,x_M\}\,.
  \]

  Let us now show that the finitely many $\widehat{d}_w$-open balls of the form
  \[
   \{U\in\mathcal{C}_w(\cH)\,|\,\widehat{d}_w(U,Z)<\varepsilon\}\,,
  \]
 %$\{U\in\mathcal{C}_w(\cH)\,|\,\widehat{d}_w(U,Z)<\varepsilon\}$ 
centred at some $Z\subset\mathcal{Z}_M$, actually cover $\mathcal{C}_w(\cH)$. Pick $U\in\mathcal{C}_w(\cH)$: as $U\subset B_\cH$, $U$ intersects some of the balls $\mathfrak{B}_w(x_n,\varepsilon)$, so let $Z_U\subset\mathcal{Z}_M$ be the collection of the corresponding centres of such balls. Thus, $U\subset Z_U(\varepsilon)$ and $U\cap \mathfrak{B}_w(x,\varepsilon)\neq \emptyset$ for any $x\in Z_U$. The last two properties are precisely the assumption of Lemma \ref{lem:38}(iii), that then implies $\widehat{d}_w(U,Z_U)<\varepsilon$. In conclusion, each $U\in \mathcal{C}_w(\cH)$ belongs to the $\widehat{d}_w$-open ball centred at $Z_U$ and with radius $\varepsilon$, and irrespectively of $U$ the number of such balls is finite, thus realising a finite cover of $\mathcal{C}_w(\cH)$.  
 \end{proof}

 \begin{proof}[Proof of Theorem \ref{thm:weak-gap-properties-and-completeness-on-ball}(v)]
  Both $f(U_n)$ and $f(U)$ are weakly closed, hence also weakly compact subsets of $B_\cH$. In particular it makes sense to evaluate $\widehat{d}_w(f(U_n),f(U))$.
  
  We start with proving that $d_w(f(U_n),f(U))\xrightarrow{n\to\infty} 0$.
  Let $\varepsilon>0$. The weakly open $\varepsilon$-expansion $f(U)(\varepsilon)$ of $f(U)$ (see \eqref{eq:epsexp} above) is weakly open in $B_\cH$, namely open in the relative topology of $B_\cH$ induced by the weak topology of $\cH$. By weak continuity, $f^{-1}(f(U)(\varepsilon))$ too is weakly open in $B_\cH$, and in fact it is a relatively open neighbourhood of $U$, for $f(U)\subset f(U)(\varepsilon)$ $\Rightarrow$ $U\subset f^{-1}(f(U)(\varepsilon))$. The set $B_\cH\setminus f^{-1}(f(U)(\varepsilon))$ is therefore weakly closed and hence weakly compact in $B_\cH$, implying that from any point $u\in U$ one has a notion of weak metric distance between $u$ and $B_\cH\setminus f^{-1}(f(U)(\varepsilon))$. So set 
  \[
   \widetilde{\varepsilon}\;:=\;\inf_{u\in U}\inf\Big\{\|u-z\|_w\,\Big|\,z\in B_\cH\setminus f^{-1}(f(U)(\varepsilon))\Big\}\,.
  \]
  It must be $\widetilde{\varepsilon}>0$, otherwise there would be a common point in $U$ and $B_\cH\setminus f^{-1}(f(U)(\varepsilon))$ (owing to the weak closedness of the latter). Thus, any weakly open expansion of $U$ up to   
  $U(\widetilde{\varepsilon})$ is surely contained in $f^{-1}(f(U)(\varepsilon))$, whence also $f(U(\widetilde{\varepsilon}))\subset f(U)(\varepsilon)$. Now, as $U_n\xrightarrow{\widehat{d}_w}U$, there is $n_\varepsilon\in\mathbb{N}$ (in fact depending on $\widetilde{\varepsilon}$, and therefore on $\varepsilon$) such that $d_w(U_n,U)<\widetilde{\varepsilon}$ for all $n\geqslant n_\varepsilon$: then (Lemma \ref{lem:38}(ii)) $U_n\subset U(\widetilde{\varepsilon})$ for all $n\geqslant n_\varepsilon$. As a consequence, for all $n\geqslant n_\varepsilon$, $f(U_n)\subset f(U(\widetilde{\varepsilon}))\subset f(U)(\varepsilon)$. Using again Lemma \ref{lem:38}(ii), $d_w(f(U_n),f(U))<\varepsilon$ for all $n\geqslant n_\varepsilon$, meaning that $d_w(f(U_n),f(U))\xrightarrow{n\to\infty} 0$.
  
  Let us now turn to proving that $d_w(f(U),f(U_n))\xrightarrow{n\to\infty} 0$. Assume for contradiction that, up to passing to a subsequence, still denoted with $(U_n)_{n\in\mathbb{N}}$, there is $\varepsilon_0>0$ such that $d_w(f(U),f(U_n))\geqslant \varepsilon_0$ $\forall n\in\mathbb{N}$. With respect to such $\varepsilon_0$, as proved in the first part, 
  %there is $\widetilde{\varepsilon}_0>0$ such that $f(U(\widetilde{\varepsilon}_0))\subset f(U)(\varepsilon_0)$, and 
  there is $n_{\varepsilon_0}\in\mathbb{N}$ such that $f(U_n)\subset f(U)(\varepsilon_0)$ $\forall n\geqslant n_{\varepsilon_0}$. For any such $n\geqslant n_{\varepsilon_0}$, on account of Lemma \ref{lem:38}(iii) one deduces from the latter two properties, namely $d_w(f(U),f(U_n))\geqslant \varepsilon_0$ and $f(U_n)\subset f(U)(\varepsilon_0)$, that there is $y_n\in f(U)$ such that $f(U_n)\cap\mathfrak{B}_w(y_n,\varepsilon_0)=\emptyset$, whence also
  \[
   U_n\cap f^{-1}(\mathfrak{B}_w(y_n,\varepsilon_0))\,=\,\emptyset\,.
  \]
%$U_n\cap f^{-1}(\mathfrak{B}_w(y_n,\varepsilon_0))=\emptyset$.  
  From this condition we want now to construct a sufficiently small weak open ball of a point $u\in U$ that is disjoint from all the $U_n$'s as well.  
  The sequence $(u^{(n)})_{n=n_{\varepsilon_0}}^\infty$ with each $u^{(n)}\in U$ such that $f(u^{(n)})=y_n$, owing to the weak compactness of $U$, has a weakly convergent subsequence to some $u\in U$. (The superscript in $u^{(n)}$ is to warn that each $u^{(n)}$ belongs to $U$, not to $U_n$.) So, up to further refinement, $u^{(n)}\rightharpoonup u$ in $U$, and by weak continuity $y_n=f(u^{(n)})\rightharpoonup f(u)=:y$. The latter convergence implies that, eventually in $n$, say, $\forall n\geqslant m_{\varepsilon_0}$ for some $m_{\varepsilon_0}\in\mathbb{N}$, $\mathfrak{B}_w(y,\frac{1}{2}\varepsilon_0)\subset\mathfrak{B}_w(y_n,\varepsilon_0)$. In view of the disjointness condition above, one then deduces
  \[
   U_n\cap f^{-1}(\mathfrak{B}_w(y,{\textstyle\frac{1}{2}}\varepsilon_0))\,=\,\emptyset\qquad \forall n\geqslant m_{\varepsilon_0}\,.
  \]
  As $f^{-1}(\mathfrak{B}_w(y,{\textstyle\frac{1}{2}}\varepsilon_0))$ above is an open neighbourhood of $u\in U$ in the relative weak topology of $B_\cH$ (weak continuity of $f$), it contains a ball $\mathfrak{B}_w(u,\varepsilon_1)$ around $u$ for some radius $\varepsilon_1>0$, whence
   \[
   U_n\cap \mathfrak{B}_w(u,\varepsilon_1)\,=\,\emptyset\qquad \forall n\geqslant m_{\varepsilon_0}\,.
  \]
  On account of Lemma \ref{lem:38}(i), this implies $d_w(U,U_n)\geqslant\varepsilon_1$ $\forall n\geqslant m_{\varepsilon_0}$. However, this contradicts the assumption $\widehat{d}_w(U,U_n)\to 0$.  
 \end{proof}

 \section{Weak gap metric for linear subspaces}
 
 Our primary interest is to exploit the $\widehat{d}_w$-convergence for closed subspaces of $\cH$, and ultimately for Krylov subspaces, in the sense of the convergence naturally induced by the convergence of the corresponding unit balls as elements of $(\mathcal{C}_w(\cH),\widehat{d}_w)$.

 In other words, given two closed subspaces $U,V\subset \cH$, by definition we identify 
  \begin{equation}\label{eq:metriclift}
  \widehat{d}_w(U,V)\;\equiv\;\widehat{d}_w(B_U,B_V)
 \end{equation}
 with the r.h.s.~defined in \eqref{eq:def-dw-in-ball}, since $B_U,B_V\in \mathcal{C}_w(\cH)$. Analogously, given 
 $U$ and a sequence $(U_n)_{n\in\mathbb{N}}$, all closed subspaces of $\cH$, we write $U_n\xrightarrow{\widehat{d}_w}U$ to mean that $B_{U_n}\xrightarrow{\widehat{d}_w}B_U$ in the sense of the definition given in the previous Section. This provides a metric topology and a notion of convergence on the set 
 \begin{equation}
  \mathscr{S}(\cH)\;:=\;\{ \textrm{closed linear subspaces of $\cH$} \}\,.
 \end{equation}
 By linearity, the closedness of each subspace of $\cH$ is equivalently meant in the $\cH$-norm or in the weak topology. (Recall, however, that the weak topology on $\cH$ is \emph{not} induced by the norm $\|\;\|_w$, as this is only the case in $B_\cH$.)

 \begin{lemma}
  The set $(\mathscr{S}(\cH),\widehat{d}_w)$ is a metric space.
 \end{lemma}

 \begin{proof}
  Positivity and triangular inequality are obvious from \eqref{eq:metriclift} and Lemma \ref{lem:inclusion-and-triangular}. Last, to deduce from $\widehat{d}_w(U,V)=0$ that $U=V$, one observes that $\widehat{d}_w(B_U,B_V)=0$ and hence $B_U=B_V$. If $u\in U$, then $u/\|u\|\in B_U=B_V$, whence by linearity $u\in V$, thus, $U\subset V$. Exchanging the role of the two subspaces, also $V\subset U$.  
 \end{proof}

 The metric space $(\mathscr{S}(\cH),\widehat{d}_w)$ contains in particular the closures of Krylov subspaces, and monitoring the distance between two such subspaces in the $\widehat{d}_w$-metric turns out to be informative in many respects. Unfortunately there is a major drawback, for:

 \begin{lemma}
  The metric space $(\mathscr{S}(\cH),\widehat{d}_w)$ is not complete.
 \end{lemma}
 
 \begin{proof}
  It is enough to provide an example of $\widehat{d}_w$-Cauchy sequence in $\mathscr{S}(\cH)$ that does not converge in $\mathscr{S}(\cH)$. So take $\cH=\ell^2(\mathbb{N})$, with the usual canonical orthonormal basis $(e_n)_{n\in\mathbb{N}}$. For $n\in\mathbb{N}$, set $U_n:=\mathrm{span}\{e_1+e_n\}\subset\mathscr{S}(\cH)$.

  Let us show first of all that the sequence $(U_n)_{n\in\mathbb{N}}$ is $\widehat{d}_w$-Cauchy, i.e., that the corresponding unit balls form  a Cauchy sequence $(B_{U_n})_{n\in\mathbb{N}}$ in the metric space $(\mathcal{C}_w,\widehat{d}_w)$. A generic $u\in B_{U_n}$ has the form $u=\alpha_u(e_1+e_n)$ for some $\alpha_u\in\mathbb{C}$ with $|\alpha_u|\leqslant\frac{1}{\sqrt{2}}$. Therefore,
  \[
   \inf_{v\in B_{U_m}}\|u-v\|_w\;\leqslant\;\|\alpha_u(e_1+e_n)-\alpha_u(e_1+e_m)\|_w\;\leqslant\;\frac{1}{\sqrt{2}}\|e_n-e_m\|_w\,,
  \]
 the first inequality following from the concrete choice $v=\alpha_u(e_1+e_m)\in B_{U_m}$. Using the above estimate and the fact that $e_n\rightharpoonup0$, and hence $(e_n)_{n\in\cH}$ is Cauchy in $\cH$, one deduces
 \[
  d_w(B_{U_n},B_{U_m})\;=\;\sup_{u\in U_n}\inf_{v\in B_{U_m}}\|u-v\|_w\;\leqslant\;\frac{1}{\sqrt{2}}\|e_n-e_m\|_w\;\xrightarrow{n,m\to\infty}\;0\,.
 \]
  Inverting $n$ and $m$ one also finds $d_w(B_{U_m},B_{U_n})\xrightarrow{n,m\to\infty}0$. The Cauchy property is thus proved.

  On account of the completeness of $(\mathcal{C}_w,\widehat{d}_w)$ (Theorem \ref{thm:weak-gap-properties-and-completeness-on-ball}(ii)), $B_{U_n}\xrightarrow{\widehat{d}_w}B$ for some $B\in\mathcal{C}_w$. Next, let us show that there is no closed subspace $U\subset\cH$ with $B_U=B$, which prevents the sequence $(U_n)_{n\in\mathbb{N}}$ to converge in $(\mathscr{S}(\cH),\widehat{d}_w)$. To this aim, we shall show that although the line segment
  \[
   \{\beta e_1\,|\,\beta\in\mathbb{C}\,,\;|\beta|\leqslant{\textstyle\frac{1}{\sqrt{2}}}\}
  \]
  is entirely contained in $B$, however $e_1\notin B$: this clearly prevents $B$ to be the unit ball of a linear subspace.
  Assume for contradiction that $e_1\in B$; then, owing to Theorem \ref{thm:weak-gap-properties-and-completeness-on-ball}(iii) (see formula \eqref{eq:limit-U} therein), $e_1\leftharpoonup u_n$ for a sequence $(u_n)_{n\in\mathbb{N}}$ with $u_n\in B_{U_n}$. In fact, weak approximants from $B_\cH$ of points of the unit sphere $S_\cH$ are necessarily also norm approximants: explicitly, owing to weak convergence, the sequence $(u_n)_{n\in\mathbb{N}}$ is norm lower semi-continuous, thus,
  \[
   1\;=\;\|e_1\|\;\leqslant\;\liminf_{n\to\infty}\|u_n\|\;\leqslant\;1\,,\qquad\textrm{whence}\qquad \lim_{n\to\infty}\|u_n\|\;=\;1\,;
  \]
  then, since $u_n\rightharpoonup e_1$ and $\|u_n\|\to\|e_1\|$, one has $u_n\to x$ in the $\cH$-norm. As a consequence, writing $u_n=\alpha_n(e_1+e_n)$ for a suitable $\alpha_n\in\mathbb{C}$ with $|\alpha_n|\leqslant\frac{1}{\sqrt{2}}$, one has $|\alpha_n|\sqrt{2}=\|u_n\|\to 1$, whence $|\alpha_n|\to\frac{1}{\sqrt{2}}$. This implies though that 
  \[
   \|u_n-e_1\|^2\;=\;|\alpha_n(e_1+e_n)-e_1\|^2\;=\;|1-\alpha_n|^2+|\alpha_n|^2
  \]
  cannot vanish as $n\to\infty$, a contradiction. Therefore, $e_1\notin B$.

  On the other hand, for any $\beta\in\mathbb{C}$ with $|\beta|\leqslant{\textstyle\frac{1}{\sqrt{2}}}$, $B_{U_n}\ni\beta(e_1+e_n)\rightharpoonup\beta e_1$, which by Theorem \ref{thm:weak-gap-properties-and-completeness-on-ball}(iii) means that $\beta e_1\in B$.  
 \end{proof}

 Despite the lack of completeness, the metric $\widehat{d}_w$ in  $\mathscr{S}(\cH)$ displays useful properties for our purposes. The first is the counterpart of Theorem \ref{thm:weak-gap-properties-and-completeness-on-ball}(iii).

 \begin{proposition}
  Let $\cH$ be a separable Hilbert space and assume that $U_n\xrightarrow{\widehat{d}_w}U$ as $n\to\infty$ for some $(U_n)_{n\in\mathbb{N}}$ and $U$ in  $\mathscr{S}(\cH)$. Then
  \begin{equation}\label{eq:limit-U-subspace}
    U\;=\;\{u\in \cH\,|\,u_n\rightharpoonup u\textrm{ for a sequence }(u_n)_{n\in\mathbb{N}} \textrm{ with } u_n\in U_n\}\,.
   \end{equation}
 \end{proposition}

 \begin{proof}
  Call temporarily
  \[
   \widehat{U}\;:=\;\{u\in \cH\,|\,u_n\rightharpoonup u\textrm{ for a sequence }(u_n)_{n\in\mathbb{N}} \textrm{ with } u_n\in U_n\}\,,
  \]
  so that the proof consists of showing that $\widehat{U}=U$.
  From Theorem \ref{thm:weak-gap-properties-and-completeness-on-ball}(iii) we know that
  \[
   B_{U_n}\,\xrightarrow{\widehat{d}_w}\,B_U\;=\;\{\widetilde{u}\in B_{\cH}\,|\,\widetilde{u}_n\rightharpoonup \widetilde{u}\textrm{ for a sequence }(\widetilde{u}_n)_{n\in\mathbb{N}} \textrm{ with } \widetilde{u}_n\in B_{U_n}\}\,.
  \]
  So now if $u\in U$, then $B_U\ni u/\|u\|\leftharpoonup \widetilde{u}_n$ for some $(\widetilde{u}_n)_{n\in\mathbb{N}}$ with $\widetilde{u}_n\in B_{U_n}$, whence $u\leftharpoonup \|u\|\widetilde{u}_n\in U_n$, meaning that $u\in\widehat{U}$. Conversely, if $u\in\widehat{U}$, and hence $u_n\rightharpoonup u$ for some $(u_n)_{n\in\mathbb{N}}$ with $u_n\in U_n$, then by uniform boundedness $\|u_n\|\leqslant \kappa$ $\forall n\in\mathbb{N}$ for some $\kappa>0$, and by lower semi-continuity of the norm along the limit $\|u\|\leqslant\kappa$ as well. As a consequence, $B_{U_n}\ni \kappa^{-1}u_n\rightharpoonup \kappa^{-1} u$, meaning that $\kappa^{-1} u\in B_U$ and therefore $u\in U$.  
 \end{proof}

 Other relevant features of the  $\widehat{d}_w$-metric will be worked out in the next Section in application to Krylov subspaces.

 \section{Krylov perturbations in the weak gap metric}\label{sec:Kry-perturb-and-wgmetric}

 We are mainly concerned with controlling how close two (closures of) Krylov subspaces $\mathcal{K}\equiv\overline{\mathcal{K}(A,g)}$ and $\mathcal{K}'\equiv\overline{\mathcal{K}(A',g')}$ are within the metric space $(\mathscr{S}(\cH),\widehat{d}_w)$ of closed subspaces of the separable Hilbert space $\cH$ with the weak gap metric $\widehat{d}_w$, for given $A,A'\in\mathcal{B}(\cH)$ and $g,g'\in\cH$.

 \subsection{Preliminary properties}~
 
 A first noticeable feature, that closes the problem left open as one initial motivation in Section \ref{sec:weakgapmetric}, is the $\widehat{d}_w$-convergence of the finite-dimensional Krylov subspace to the corresponding closed Krylov subspace.

 \begin{lemma}\label{lem:KriNtoKri}
  Let $A\in\mathcal{B}(\cH)$ and $g\in\cH$ for an infinite-dimensional, separable Hilbert space $\cH$. Then 
  \begin{equation}
   \widehat{d}_w\big( \mathcal{K}_N(A,g),\overline{\mathcal{K}(A,g)}\big)\;\xrightarrow{N\to\infty}\;0\,,
  \end{equation}
  with the two spaces defined, respectively, in \eqref{eq:defKrylov} and \eqref{eq:defKrylov-N}.
 \end{lemma}

 This means that the $\widehat{d}_w$-metric provides the appropriate language to measure the distance between $\mathcal{K}_N(A,g)$ and $\overline{\mathcal{K}(A,g)}$ in an informative way: $\mathcal{K}_N(A,g)\xrightarrow{\widehat{d}_w}\overline{\mathcal{K}(A,g)}$, whereas we saw that it is false in general that $\mathcal{K}_N(A,g)\xrightarrow{\widehat{d}}\overline{\mathcal{K}(A,g)}$.

 For the simple proof of this fact, and for later purposes, it is convenient to work out the following useful construction.

  \begin{lemma}\label{lem:approx-le-1}
  Let $\cH$ be a separable Hilbert space, and let $A\in\mathcal{B}(\cH)$, $g\in\cH$. Set $\mathcal{K}:=\overline{\mathcal{K}(A,g)}$.
  \begin{itemize}
   \item[(i)] For every $x\in B_{\mathcal{K}}$ here exists a sequence $(u_n)_{n\in\mathbb{N}}$ in $\mathcal{K}(A,g)$ such that $\|u_n\|<1$ $\forall n\in\mathbb{N}$ and $u_n\xrightarrow[n\to\infty]{\|\,\|}x$.
   \item[(ii)] For every $\varepsilon>0$ there is a cover of $B_\mathcal{K}$ consisting of finitely many weakly open balls $\mathfrak{B}_w(x_1,\varepsilon),\dots,\mathfrak{B}_w(x_{M},\varepsilon)$ for some $x_1,\dots, x_M\in B_\mathcal{K}$ and $M\in\mathbb{N}$ all depending on $\varepsilon$. Moreover, each centre $x_j\in B_\mathcal{K}$ has an approximant $p_j(A)g$ for some polynomial $p_j$ on $\mathbb{R}$ with $\|p_j(A)g\|<1$ and $\|p_j(A)g-x_j\|\leqslant\varepsilon$ $\forall j\in\{1,\dots,M\}$.   
  \end{itemize}
   \end{lemma}

  \begin{proof}  
   As $x\in\overline{\mathcal{K}(A,g)}$, then $\widetilde{u}_n\xrightarrow[]{\|\,\|} x$ for some sequence $(\widetilde{u}_n)_{n\in\mathbb{N}}$ in $\mathcal{K}(A,g)$. In particular, $\|\widetilde{u}_n\|\to\|x\|$, and it is not restrictive to assume $\|\widetilde{u}_n\|>0$ for all $n$. Therefore,
   \[
    u_n\;:=\;\frac{n-1}{n}\,\frac{\|x\|}{\|\widetilde{u}_n\|}\,\widetilde{u}_n\;\in\;\mathcal{K}(A,g)\,,\qquad\textrm{and}\qquad \|u_n\|\;<\;\|\widetilde{u}_n\|\;\leqslant\;1\,,
   \]
  and obviously $u_n\xrightarrow[n\to\infty]{\|\,\|}x$. This proves part (i). Concerning part (ii), the existence of such cover follows from the compactness of $(B_\cH,\varrho_w)$ and the closure of $B_\mathcal{K}$ in $B_\cH$. The approximants $p_n(A)g$ are then found based on part (i).
    \end{proof}

  \begin{proof}[Proof of Lemma \ref{lem:KriNtoKri}]
   Let us use the shorthand $\mathcal{K}_N\equiv\mathcal{K}_N(A,g)$ and $\mathcal{K}\equiv\mathcal{K}(A,g)$.   
   As $\mathcal{K}_N\subset\mathcal{K}$, then $d_w(\mathcal{K}_N,\mathcal{K})=0$ (see \eqref{eq:dwinclusion}, Lemma \ref{lem:inclusion-and-triangular} above), so only the limit $d_w(\mathcal{K},\mathcal{K}_N)\equiv d_w(B_{\mathcal{K}},B_{\mathcal{K}_N})\to 0$ is to be checked.
   
   For $\varepsilon>0$ take the finite open $\varepsilon$-cover of $B_\mathcal{K}$ constructed in Lemma \ref{lem:approx-le-1} with centres $x_1,\dots,x_M$ and Krylov approximants $p_1(A)g,\dots,p_M(A)g$. Let $N_0$ be the largest degree of the $p_j$'s, thus ensuring that $p_j(A)g\in B_{\mathcal{K}_N}$ $\forall j\in\{1,\dots,M\}$ and $\forall N\geqslant N_0+1$.
   
   Now consider an arbitrary integer $N\geqslant N_0+1$ and an arbitrary $u\in B_\mathcal{K}$. The vector $u$ clearly belongs to at least one of the balls of the finite open cover above: up to re-naming the centres, it is non-restrictive to claim that $u\in\mathfrak{B}_w(x_1,\varepsilon)$, and consider the above approximant $p_1(A)g\in B_{\mathcal{K}_N}$ of the ball's centre $x_1$. Thus, $\|u-x_1\|_w<\varepsilon$ and $\|x_1-p_1(A)g\|_w\leqslant\|x_1-p_1(A)g\|\leqslant\varepsilon$. Then
   \[
    \begin{split}
     \inf_{v\in B_{\mathcal{K}_N}}\|u-v\|_w\;&\leqslant\;\|u-x_1\|_w+\|x_1-p_1(A)g\|_w+\inf_{v\in B_{\mathcal{K}_N}}\|p_1(A)g-v\|_w \\
     &<\;2\varepsilon
    \end{split}
   \]
  whence also $d_w(B_{\mathcal{K}},B_{\mathcal{K}_N})=\displaystyle\sup_{u\in B_{\mathcal{K}}}\inf_{v\in B_{\mathcal{K}_N}}\|u-v\|_w\leqslant 2\varepsilon$.
  \end{proof}

  Despite the encouraging property stated in Lemma \ref{lem:KriNtoKri}, one soon learns that the sequences of (closures of) Krylov subspaces with good convergence properties of the Krylov data $A$ and/or $g$ display in general quite a diverse (including non-convergent) behaviour in the $\widehat{d}_w$-metric. This suggests that an efficient control of $\widehat{d}_w$-convergence of Krylov subspaces is only possible under suitable restrictive assumptions.

  Lemma \ref{lem:gntogKnCauchy}, Example \ref{ex:ell2toK} and Example \ref{ex:gnbut-notCauchy} below are meant to shed some light on this scenario. In particular, Lemma \ref{lem:gntogKnCauchy} establishes that the convergence $g_n\to g$ in $\cH$ is sufficient to have $d_w(\overline{\mathcal{K}(A,g)},\overline{\mathcal{K}(A,g_n)})\to 0$.

  \begin{lemma}\label{lem:gntogKnCauchy}
   Given a separable Hilbert space $\cH$ and $A\in\mathcal{B}(\cH)$, assume that $g_n\xrightarrow[n\to\infty]{\|\,\|}g$ for vectors $g,g_n\in\cH$. Set $\mathcal{K}_n\equiv\overline{\mathcal{K}(A,g_n)}$ and $\mathcal{K}\equiv\overline{\mathcal{K}(A,g)}$.
   \begin{itemize}
    \item[(i)] One has $d_w(K,K_n)\xrightarrow{n\to\infty}0$.
    \item[(ii)] From a sequence $(B_{K_n})_{n\in\mathbb{N}}$ extract, by compactness of $(\mathcal{C}_w(\cH),\widehat{d}_w)$, a convergent subsequence to some $B\in \mathcal{C}_w(\cH)$. Then $B_K\subset B$.
   \end{itemize}
  \end{lemma}

  \begin{proof}~
  
  (i) For $\varepsilon>0$ take the finite open $\varepsilon$-cover of $B_\mathcal{K}$ constructed in Lemma \ref{lem:approx-le-1} with centres $x_1,\dots,x_M$ and Krylov approximants $p_1(A)g,\dots,p_M(A)g$. In view of the finitely many conditions $\|p_j(A)g\|<1$ and $p_j(A)g_n\xrightarrow[n\to\infty]{\|\,\|}p_j(A)g$, $j\in\{1,\dots,M\}$, there is $n_\varepsilon\in\mathbb{N}$ such that $\|p_j(A)g_n-p_j(A)g\|\leqslant\varepsilon$ and $\|p_j(A)g_n\|<1$  for all $n\geqslant n_\varepsilon$ and $j\in\{1,\dots,M\}$.

    Take $u\in B_{\mathcal{K}}$. Up to re-naming the centres of the cover's balls, $\|u-x_1\|_w<\varepsilon$, $\|x_1-p_1(A)g\|_w\leqslant\varepsilon$, $\|p_1(A)g_n-p_1(A)g\|\leqslant\varepsilon$, and $\|p_1(A)g_n\|<1$ $\forall n\geqslant n_\varepsilon$. Then, for any $n\geqslant n_\varepsilon$,
    \[
     \begin{split}
      \inf_{v\in B_{\mathcal{K}_n}}\|u-v\|_w\;&\leqslant\;\|u-x_1\|_w+\|x_1-p_1(A)g\|_w+\inf_{v\in B_{\mathcal{K}_n}}\|p_1(A)g-v\|_w \\
      &\leqslant\; 2\varepsilon+\|p_1(A)g-p_1(A)g_n\|_w \;\leqslant\;3\varepsilon\,,
     \end{split}
    \]
    whence also, for $n\geqslant n_\varepsilon$,
    \[
     d_w(\mathcal{K},\mathcal{K}_n)\;\equiv\;d_w(B_\mathcal{K},B_{\mathcal{K}_n})\;=\;\sup_{u\in B_\mathcal{K}}\inf_{v\in B_{\mathcal{K}_n}}\|u-v\|_w\;\leqslant\;3\varepsilon\,.
    \]
    This means precisely that $d_w(\mathcal{K},\mathcal{K}_n)\to 0$.

   (ii) Rename the extracted subsequence again as $(B_{K_n})_{n\in\mathbb{N}}$, so that $B_{K_n}\xrightarrow{\widehat{d}_w}B$. On account of \eqref{eq:d-triangular} (Lemma \ref{lem:inclusion-and-triangular}), 
   \[
    d_w(B_K,B)\;\leqslant\;d_w(B_K,B_{K_n})+d_w(B_{K_n},B)\,.
   \]
   Since $d_w(B_{K_n},B)\leqslant\widehat{d}_w(B_{K_n},B)\to 0$ by assumption, and $d_w(B_K,B_{K_n})\to 0$ as established in part (i), then $d_w(B_K,B)=0$. Owing to \eqref{eq:dwinclusion} (Lemma \ref{lem:inclusion-and-triangular}), this implies $B_K\subset B$.   
  \end{proof}

  \begin{example}\label{ex:ell2toK}
   In general, the assumptions of Lemma \ref{lem:gntogKnCauchy} are \emph{not enough} to guarantee that also $d_w(\mathcal{K}_n,\mathcal{K})\to 0$ and hence $\mathcal{K}_n\xrightarrow{\widehat{d}_x}\mathcal{K}$. Consider for instance $\cH=\ell^2(\mathbb{N})$ and the right-shift operator $A\equiv R$, acting as $Re_k=e_{k+1}$ on the canonical basis $(e_k)_{k\in\mathbb{N}}$. As in Example \ref{ex:gainloss-with-shift}, $R$ admits a dense of cyclic vectors, as well as a dense of non-cyclic vectors: so, with respect to the general setting of Lemma \ref{lem:gntogKnCauchy}, take now $g$ to be \emph{non-cyclic}, say, $g=e_2$, and $(g_n)_{n\in\mathbb{N}}$ to be a sequence of $\cH$-norm approximants of $g$ that are all \emph{cyclic}. Concerning the subspaces $\mathcal{K}_n:=\overline{\mathcal{K}(R,g_n)}$ and $\mathcal{K}:=\overline{\mathcal{K}(R,g)}$, $\mathcal{K}_n=\cH$ $\forall n\in\mathbb{N}$ by cyclicity, and $\mathcal{K}=\{e_1\}^\perp\varsubsetneq\cH$. As $\mathcal{K}\subset\mathcal{K}_n$, then $d_w(\mathcal{K},\mathcal{K}_n)= 0$, a conclusion consistent with Lemma \ref{lem:gntogKnCauchy}, for $(\mathcal{K}_n)_{n\in\mathbb{N}}$ is obviously $\widehat{d}_w$-Cauchy and Lemma \ref{lem:gntogKnCauchy} implies $d_w(\mathcal{K},\mathcal{K}_n)\to 0$. On the other hand,
   \[
    d_w(B_{\mathcal{K}_n},B_{\mathcal{K}})\;=\;\sup_{u\in B_{\mathcal{K}_n}}\inf_{v\in B_{\mathcal{K}}}\|u-v\|_w\;\geqslant\;\inf_{v\in B_{\mathcal{K}}}\|e_1-v\|_w\;>\;0\,,
   \]
   which prevents $d_w(\mathcal{K}_n,\mathcal{K})$ to vanish with $n$.
  \end{example}

  \begin{example}\label{ex:gnbut-notCauchy}
   In general, with respect to the setting of Lemma \ref{lem:gntogKnCauchy} and Example \ref{ex:ell2toK}, the sole convergence $g_n\xrightarrow[]{\|\,\|}g$ is \emph{not enough} to guarantee that $(\mathcal{K}_n)_{n\in\mathbb{N}}$ be $\widehat{d}_w$-Cauchy. For, again with the right-shift $R$ on $\cH=\ell^2(\mathbb{N})$, take now a sequence $(\widetilde{g}_n)_{n\in\mathbb{N}}$ of cyclic vectors for $R$ such that $\widetilde{g}_n\xrightarrow[]{\|\,\|}e_2$, and set 
   \[
    g_n\;:=\;
    \begin{cases}
     \:\widetilde{g}_n & \textrm{ for even $n$} \\
     \:e_2 & \textrm{ for odd $n$}\,.
    \end{cases}
   \]
   Thus, $g_n\xrightarrow[]{\|\,\|}g:=e_2$. For even $n$, $\mathcal{K}_n:=\overline{\mathcal{K}(R,g_n)}=\cH$ and $\mathcal{K}_{n+1}=\{e_1\}^\perp$, whence
   \[
     d_w(B_{\mathcal{K}_n},B_{\mathcal{K}_{n+1}})\;=\;\sup_{u\in B_{\mathcal{K}_n}}\inf_{v\in B_{\mathcal{K}_{n+1}}}\|u-v\|_w\;\geqslant\;\inf_{v\in B_{\mathcal{K}_{n+1}}}\|e_1-v\|_w\;>\;0\,,
   \]
   which prevents $(\mathcal{K}_m)_{m\in\mathbb{N}}$ to be $\widehat{d}_w$-Cauchy.
  \end{example}

 \subsection{Existence of $\widehat{d}_w$-limits. Krylov inner approximability.}\label{sec:innerapprox}~

 Based on the examples discussed above, one is to expect a variety sufficient conditions ensuring the convergence of a sequence of (closures of) Krylov subspaces to a (closure of) Krylov subspace. In this Subsection we discuss one mechanism of convergence that is meaningful in our context of Krylov perturbations.

%  The next example illustrates the situation.

 \begin{proposition}\label{prop:inner-approxim}
  Let $\cH$ be a separable Hilbert space, $A\in\mathcal{B}(\cH)$, and $g\in\cH$. Assume further that there is a sequence $(g_n)_{n\in\mathbb{N}}$ such that
  \begin{equation}\label{eq:inner-approx}
   g_n\in \overline{\mathcal{K}(A,g)}\;\;\;\forall n\in\mathbb{N}\qquad\textrm{and}\qquad g_n\xrightarrow[n\to\infty]{\|\,\|}g\,.
  \end{equation}
  Then $\overline{\mathcal{K}(A,g_n)}\;\xrightarrow{\widehat{d}_w}\;\overline{\mathcal{K}(A,g)}$.  
 \end{proposition}

 \begin{proof}
 %[Proof of Proposition \ref{prop:inner-approxim}]
 Let us use the shorthand $\mathcal{K}_n\equiv\overline{\mathcal{K}(A,g_n)}$, $\mathcal{K}\equiv\overline{\mathcal{K}(A,g)}$.  As $\mathcal{K}_n\subset\mathcal{K}$, then $d_w(\mathcal{K}_n,\mathcal{K})=0$. As $g_n\to g$ in $\cH$, then $d_w(\mathcal{K},\mathcal{K}_n)\to 0$ (Lemma \ref{lem:gntogKnCauchy}). Thus, $\mathcal{K}_n\xrightarrow{\widehat{d}_w}\mathcal{K}$.   
%  and one only needs to prove  .
%  
%  
%      For $\varepsilon>0$ take the finite open $\varepsilon$-cover of $B_\mathcal{K}$ constructed in Lemma \ref{lem:approx-le-1} with centres $x_1,\dots,x_M$ and Krylov approximants $p_1(A)g,\dots,p_M(A)g$. In view of the finitely many conditions $\|p_j(A)g\|<1$ and $p_j(A)g_n\xrightarrow[n\to\infty]{\|\,\|}p_j(A)g$, $j\in\{1,\dots,M\}$, there is $n_\varepsilon\in\mathbb{N}$ such that $\|p_j(A)g_n-p_j(A)g\|\leqslant\varepsilon$ and $\|p_j(A)g_n\|<1$  for all $n\geqslant n_\varepsilon$ and $j\in\{1,\dots,M\}$.
%     
%     
%     Take $u\in B_{\mathcal{K}}$. Up to re-naming the centres of the cover's balls, $\|u-x_1\|_w<\varepsilon$, $\|x_1-p_1(A)g\|_w\leqslant\varepsilon$, $\|p_1(A)g_n-p_1(A)g\|\leqslant\varepsilon$, and $\|p_1(A)g_n\|<1$ $\forall n\geqslant n_\varepsilon$. Then, for any $n\geqslant n_\varepsilon$,
%     \[
%      \begin{split}
%       \inf_{v\in B_{\mathcal{K}_n}}\|u-v\|_w\;&\leqslant\;\|u-x_1\|_w+\|x_1-p_1(A)g\|_w+\inf_{v\in B_{\mathcal{K}_n}}\|p_1(A)g-v\|_w \\
%       &\leqslant\; 2\varepsilon+\|p_1(A)g-p_1(A)g_n\|_w \;\leqslant\;3\varepsilon\,,
%      \end{split}
%     \]
%     whence also, for $n\geqslant n_\varepsilon$,
%     \[
%      d_w(\mathcal{K},\mathcal{K}_n)\;\equiv\;d_w(B_\mathcal{K},B_{\mathcal{K}_n})\;=\;\sup_{u\in B_\mathcal{K}}\inf_{v\in B_{\mathcal{K}_n}}\|u-v\|_w\;\leqslant\;3\varepsilon\,.
%     \]
%     This means precisely that $d_w(\mathcal{K},\mathcal{K}_n)\to 0$, whence finally $\mathcal{K}_n\xrightarrow{\widehat{d}_w}\mathcal{K}$.   
\end{proof}

 The above convergence $\overline{\mathcal{K}(A,g_n)}\xrightarrow{\widehat{d}_w}\overline{\mathcal{K}(A,g)}$, in view of condition \eqref{eq:inner-approx},  expresses the ``inner approximability'' of $\overline{\mathcal{K}(A,g)}$. In fact, $\overline{\mathcal{K}(A,g_n)}\subset\overline{\mathcal{K}(A,g)}$. Condition \eqref{eq:inner-approx} includes also the case of approximants $g_n$ from $\mathcal{K}(A,g)$ or also from $\mathcal{K}_n(A,g)$ (the $n$-th order Krylov subspace \eqref{eq:defKrylov-N}). For instance, set
 \[
   g_n\;:=\;\sum_{k=0}^{n-1} \frac{1}{\:n^{2k}\|A\|_{\mathrm{op}}^k} A^k g\;\in\;\mathcal{K}_n(A,g)\,,\qquad n\in\mathbb{N}\,,
  \]
 and as 
 \[
  \|g-g_n\|\;\leqslant\;\sum_{k=1}^{n-1} \frac{\|A^k g\|}{\:n^{2k}\|A\|_{\mathrm{op}}^k}\;\leqslant\;\|g\|\sum_{k=1}^{n-1} \frac{1}{\:n^{2k}}\;\leqslant\;\frac{\|g\|}{n}\,,
 \]
  then $\overline{\mathcal{K}(A,g)}\supset\mathcal{K}(A,g)\supset \mathcal{K}_n(A,g)\ni g_n\to g$ in $\cH$.
 
%  
%  \begin{example}
%   $A\in\mathcal{B}(\cH)$ and $g\in\cH$ be given.
%   \begin{itemize}
%    \item[(i)] Set 
%    \[
%    g_N\;:=\;\sum_{k=0}^{N-1} \frac{1}{\:N^{2k}\|A\|_{\mathrm{op}}^k} A^k g\;\in\;\mathcal{K}_N(A,g)\,,\qquad N\in\mathbb{N}\,.
%   \]
%  As
%  \[
%   \|g-g_N\|\;\leqslant\;\sum_{k=1}^{N-1} \frac{\|A^k g\|}{\:N^{2k}\|A\|_{\mathrm{op}}^k}\;\leqslant\;\|g\|\sum_{k=1}^{N-1} \frac{1}{\:N^{2k}}\;\leqslant\;\frac{\|g\|}{N}\,,
%  \]
%   then $\overline{\mathcal{K}(A,g)}\supset\mathcal{K}(A,g)\supset \mathcal{K}_N(A,g)\ni g_N\to g$ in $\cH$-norm.
%   \item[(ii)] Assume in addition that $A$ is self-adjoint and denote by $\mu_g^{(A)}$ the scalar spectral measure associated with $A$ and $g$. As we proved in \cite[Theorem 7.2]{CM-2019_ubddKrylov}, there is a canonical Hilbert space isomorphism
%   \begin{equation*}
%    L^2(\mathbb{R},\ud\mu_g^{(A)})\,\xrightarrow{\;\cong\;}\,\overline{\mathcal{K}(A,g)}\,,\qquad f\longmapsto f(A)g
%   \end{equation*}
%   with the operator $f(A)$ constructed by functional calculus. Let now $(f_n)_{n\in\mathbb{N}}$ be a sequence in $L^2(\mathbb{R},\ud\mu_g^{(A)})$ that $L^2$-converges to the function $\mathbf{1}$, for instance $f_n:=\mathbf{1}_{[-n,n]}$ or $f_n:=1-e^{-n|x|}$ (in which case $f_n\xrightarrow{L^2}\mathbf{1}$ follows by dominated convergence). Then $g_n:=f_n(A)g\in\overline{\mathcal{K}(A,g)}$ and $g_n\to g$ in $\cH$-norm.  
%   \end{itemize}
%  \end{example}
% 

\subsection{Krylov solvability along $\widehat{d}_w$-limits}~

Let us finally scratch the surface of a very central question for the present investigation, namely how a perturbation of a given inverse linear problem, that is small in $\widehat{d}_w$-sense for the corresponding Krylov subspaces, does affect the Krylov solvability.

Far from answering in general, we have at least the tools to control the following class of cases. The proof is fast, but it relies on two non-trivial toolboxes.

\begin{proposition}\label{prop:Krisolv-along-onelimit}
 Let $\cH$ be a separable Hilbert space. The following be given:
 \begin{itemize}
  \item an operator $A\in\mathcal{B}(\cH)$ with inverse $A^{-1}\in \mathcal{B}(\cH)$;
  \item a sequence $(g_n)_{n\in\mathbb{N}}$ in $\cH$ such that for each $n$ the (unique) solution $f_n:=A^{-1}g_n$ to the inverse problem $Af_n=g_n$ is a Krylov solution;
  \item a vector $g\in\cH$ such that $\overline{\mathcal{K}(A,g_n)}\xrightarrow{\widehat{d}_w}\overline{\mathcal{K}(A,g)}$ as $n\to\infty$.
 \end{itemize}
  Then the (unique) solution $f:=A^{-1}g$ to the inverse problem $Af=g$ is a Krylov solution. If in addition $g_n\to g$, respectively $g_n\rightharpoonup g$, then $f_n\to f$, respectively $f_n\rightharpoonup f$.
\end{proposition}

\begin{proof}
 As $A$ is a bounded bijection of $\cH$ with bounded inverse, $A$ is a strongly continuous and closed $\cH\to\cH$ (linear) map, and therefore it also weakly continuous and weakly closed. Up to a non-restrictive scaling one may assume that $\|A\|_{\mathrm{op}}\leqslant 1$, implying that $A$ maps $B_\cH$ into itself. The conditions of Theorem \ref{thm:weak-gap-properties-and-completeness-on-ball}(v) are therefore matched. Thus, from $\overline{\mathcal{K}(A,g_n)}\xrightarrow{\widehat{d}_w}\overline{\mathcal{K}(A,g)}$ one deduces $A\overline{\mathcal{K}(A,g_n)}\xrightarrow{\widehat{d}_w}A\overline{\mathcal{K}(A,g)}$. On the other hand, based on a result that we proved in \cite[Prop.~3.2(ii)]{CMN-2018_Krylov-solvability-bdd}, the assumption that $f_n\in\overline{\mathcal{K}(A,g_n)}$ is equivalent to $A\overline{\mathcal{K}(A,g_n)}=\overline{\mathcal{K}(A,g_n)}$. Thus, 
 $\overline{\mathcal{K}(A,g_n)}=A\overline{\mathcal{K}(A,g_n)}\xrightarrow{\widehat{d}_w}A\overline{\mathcal{K}(A,g)}$. The $\widehat{d}_w$-limit being unique, $A\overline{\mathcal{K}(A,g)}=\overline{\mathcal{K}(A,g)}$. Then, again on account of \cite[Prop.~3.2(ii)]{CMN-2018_Krylov-solvability-bdd}, $f\in\overline{\mathcal{K}(A,g_n)}$. This proves the main statement; the additional convergences of $(f_n)_{n\in\mathbb{N}}$ to $f$ are obvious. 
\end{proof}

\begin{remark}
 It is worth stressing that the the control of the perturbation in Proposition \ref{prop:Krisolv-along-onelimit}, namely the assumption $\overline{\mathcal{K}(A,g_n)}\xrightarrow{\widehat{d}_w}\overline{\mathcal{K}(A,g)}$, does not necessarily correspond to some $\cH$-norm vicinity between $g_n$ and $g$ (in Proposition \ref{prop:inner-approxim}, instead, we had discussed a case where $\overline{\mathcal{K}(A,g_n)}\xrightarrow{\widehat{d}_w}\overline{\mathcal{K}(A,g)}$ is a consequence of $g_n\to g$ in $\cH$). The following example elucidates the situation. With respect to the general setting of Proposition \ref{prop:Krisolv-along-onelimit}, consider $\cH=\ell^2(\mathbb{N})$, $A=\mathbbm{1}$, $g=0$, $g_n=e_n$ (the $n$-th canonical basis vector), and hence
 \[
  \begin{split}
   K_n\;&:=\;\overline{\mathcal{K}(A,g_n)}\;=\;\mathrm{span}\{e_n\}\,, \\
   K\;&:=\;\overline{\mathcal{K}(A,g)}\;=\;\{0\}\,.
  \end{split}
 \]
 Obviously $B_K\subset B_{K_n}$, whence $d_w(K,K_n)=0$, on account of \eqref{eq:dwinclusion} and \eqref{eq:metriclift}. On the other hand, a generic $u\in B_{K_n}$ has the form $u=\alpha e_n$ for some $|\alpha| \leqslant 1$. Therefore,
 \[
   d_w(K_n,K)\;=\;\sup_{u\in B_{K_n}}\inf_{v\in B_K} \|u - v\|_w \;=\;\sup_{u\in B_{K_n}}\|u\|_w\;\leqslant\;\|e_n\|_w \xrightarrow{\;n\to\infty\;}\;0\,.
 \]
 This shows that $\overline{\mathcal{K}(A,g_n)}\xrightarrow{\widehat{d}_w}\overline{\mathcal{K}(A,g)}$. Thus, all assumptions of Proposition \ref{prop:Krisolv-along-onelimit} are matched. However, it is false that $g_n$ converges to $g$ in norm: in this case it is only true that $g_n\rightharpoonup g$ (weakly in $\cH$), indeed $e_n\rightharpoonup 0$. 
\end{remark}

\section{Conclusions and perspectives}\label{sec:conclusions}

In retrospect, a few concluding observations are in order.

We have already elaborated in the opening Section \ref{intro} that the main perspective of this kind of investigation is to regard a perturbed inverse problem as a potentially ``easier'' source of information, including Krylov solvability, for the original, unperturbed problem, and conversely to understand when a given inverse problem looses Krylov solvability under small perturbations, that in practice would correspond to uncertainties of various sort, thus making Krylov subspace methods potentially unstable.

The evidence from Section \ref{sec:gain-loss} is that a controlled vicinity of the perturbed operator or the perturbed datum is not sufficient, alone, to decide on the above questions, for Krylov solvability may well persist, disappear, or appear in the limit when the perturbation is removed. And the idea inspiring Section \ref{sec:K-class} is that constraining the perturbation within certain classes of operators may provide the additional information needed. Thus, a first plausible research programme is to investigate what classes of operators undergo perturbations that make Krylov solvability stable.

The attempt we then made in Sections \ref{sec:weakgapmetric}-\ref{sec:Kry-perturb-and-wgmetric} is to encode the inverse problem perturbation into a convenient topology that allows to predict whether Krylov solvability persists or is washed out. On a conceptual footing this \emph{is} the appropriate approach, because we know from our previous investigation \cite{CMN-2018_Krylov-solvability-bdd} that Krylov solvability is essentially a structural property of the Krylov subspace $\mathcal{K}(A,g)$, therefore it is natural to compare Krylov subspaces in a meaningful sense. The weak gap metric for linear subspaces of $\cH$, while being encouraging in many respects ($\mathcal{K}_N\to\mathcal{K}$, inner approximability, stability under perturbations in the sense of Proposition \ref{prop:Krisolv-along-onelimit}), suffers various limitations that need be further understood (indirectly due to the lack of completeness of the $\widehat{d}_w$-metric out of the Hilbert closed unit ball, in turn due to the lack of metrisability of the weak topology out of the unit ball). It is plausible to expect, and so is our next commitment, that the informative control of the inverse problem perturbation, as far as Krylov solvability is concerned, is a combination of an efficient distance between Krylov subspaces, vicinity of operators and of data, and restriction to classes of distinguished operators.

At this stage, this preliminary investigation completes a \emph{first} cycle of study on abstract inverse linear problems, their finite-dimensional truncations and approximations, their Krylov solvability in the bounded and unbounded case, and the stability of Krylov solvability under perturbations, that we developed in our previous recent works \cite{CMN-2018_Krylov-solvability-bdd,CMN-truncation-2018,CM-Nemi-unbdd-2019,CM-2019_ubddKrylov} and in the present one.

%\section*{Acknowledgements}
% 
% 
% 
% \bibliographystyle{siam}
% \bibliography{bib_ALE}

\begin{thebibliography}{10}

\bibitem{Akhiezer-Glazman-1961-1993}
{\sc N.~I. Akhiezer and I.~M. Glazman}, {\em {Theory of linear operators in
  {H}ilbert space}}, Dover Publications, Inc., New York, 1993.
\newblock Translated from the Russian and with a preface by Merlynd Nestell,
  Reprint of the 1961 and 1963 translations, Two volumes bound as one.

\bibitem{Brezis-FA-Sob-PDE}
{\sc H.~Brezis}, {\em {Functional analysis, {S}obolev spaces and partial
  differential equations}}, {Universitext}, Springer, New York, 2011.

\bibitem{CM-2019_ubddKrylov}
{\sc N.~A. Caruso and A.~Michelangeli}, {\em {Krylov {S}olvability of
  {U}nbounded {I}nverse {L}inear {P}roblems}}, Integral Equations Operator
  Theory, 93 (2021), p.~Paper No. 1.

\bibitem{CM-Nemi-unbdd-2019}
\leavevmode\vrule height 2pt depth -1.6pt width 23pt, {\em {Convergence of the
  conjugate gradient method with unbounded operators}}, arXiv:1908.10110
  (2019).

\bibitem{CMN-2018_Krylov-solvability-bdd}
{\sc N.~A. Caruso, A.~Michelangeli, and P.~Novati}, {\em {On Krylov solutions
  to infinite-dimensional inverse linear problems}}, Calcolo, 56 (2019), p.~32.

\bibitem{CMN-truncation-2018}
\leavevmode\vrule height 2pt depth -1.6pt width 23pt, {\em {On general
  projection methods and convergence behaviours for abstract linear inverse
  problems}}, arXiv:1811.08195 (2018).

\bibitem{Daniel-1967}
{\sc J.~W. Daniel}, {\em {The conjugate gradient method for linear and
  nonlinear operator equations}}, SIAM J. Numer. Anal., 4 (1967), pp.~10--26.

\bibitem{Du-Sarkis-Schaerer-Szyld-2013}
{\sc X.~Du, M.~Sarkis, C.~E. Schaerer, and D.~B. Szyld}, {\em {Inexact and
  truncated Parareal-in-time Krylov subspace methods for parabolic optimal
  control problems}}, Electron. Trans. Numer. Anal., 40 (2013), pp.~36--57.

\bibitem{Ern-Guermond_book_FiniteElements}
{\sc A.~Ern and J.-L. Guermond}, {\em {Theory and practice of finite
  elements}}, vol.~159 of {Applied Mathematical Sciences}, Springer-Verlag, New
  York, 2004.

\bibitem{Geher-1972}
{\sc L.~Geh{\'e}r}, {\em {Cyclic vectors of a cyclic operator span the space}},
  Proc. Amer. Math. Soc., 33 (1972), pp.~109--110.

\bibitem{Gohberg-Markus-1959}
{\sc I.~C. Gohberg and A.~S. Markus}, {\em {Two theorems on the opening between
  subspaces of Banach space}}, Uspekhi Mat. Nauk., 5(89) (1959), pp.~135--140.

\bibitem{Gupta-Mukherjee-2019}
{\sc A.~K. Gupta and S.~Mukherjee}, {\em {On Hausdorff Metric Spaces}},
  arXiv:1909.07195 (2019).

\bibitem{Halmos-HilbertSpaceBook}
{\sc P.~R. Halmos}, {\em {A {H}ilbert space problem book}}, vol.~19 of
  {Graduate Texts in Mathematics}, Springer-Verlag, New York-Berlin,
  second~ed., 1982.
\newblock Encyclopedia of Mathematics and its Applications, 17.

\bibitem{Hanke-ConjGrad-1995}
{\sc M.~Hanke}, {\em {Conjugate gradient type methods for ill-posed problems}},
  vol.~327 of {Pitman Research Notes in Mathematics Series}, Longman Scientific
  \& Technical, Harlow, 1995.

\bibitem{Hansen-Illposed-1998}
{\sc P.~C. Hansen}, {\em {Rank-deficient and discrete ill-posed problems}},
  {SIAM Monographs on Mathematical Modeling and Computation}, Society for
  Industrial and Applied Mathematics (SIAM), Philadelphia, PA, 1998.
\newblock Numerical aspects of linear inversion.

\bibitem{Henrikson-1999}
{\sc J.~Henrikson}, {\em {Completeness and total boundedness of the Hausdorff
  metric}}, MITUndergrad J. Math., 1 (1999), pp.~69--80.

\bibitem{Herrero-1972}
{\sc D.~A. Herrero}, {\em {Eigenvectors and cyclic vectors for bilateral
  weighted shifts}}, Rev. Un. Mat. Argentina, 26 (1972/73), pp.~24--41.

\bibitem{Herzog-Ekkehard-2015}
{\sc R.~Herzog and E.~Sachs}, {\em {Superlinear convergence of {K}rylov
  subspace methods for self-adjoint problems in {H}ilbert space}}, SIAM J.
  Numer. Anal., 53 (2015), pp.~1304--1324.

\bibitem{Kammerer-Nashed-1972}
{\sc W.~J. Kammerer and M.~Z. Nashed}, {\em {On the convergence of the
  conjugate gradient method for singular linear operator equations}}, SIAM J.
  Numer. Anal., 9 (1972), pp.~165--181.

\bibitem{Karush-1952}
{\sc W.~Karush}, {\em {Convergence of a method of solving linear problems}},
  Proc. Amer. Math. Soc., 3 (1952), pp.~839--851.

\bibitem{Kato-perturbation}
{\sc T.~Kato}, {\em {Perturbation theory for linear operators}}, {Classics in
  Mathematics}, Springer-Verlag, Berlin, 1995.
\newblock Reprint of the 1980 edition.

\bibitem{Krein-Krasnoselskii-1947}
{\sc M.~G. Kre{\u\i}n and M.~A. Krasnosel{\cprime}ski{\u\i}}, {\em {Fundamental
  theorems on the extension of {H}ermitian operators and certain of their
  applications to the theory of orthogonal polynomials and the problem of
  moments}}, Uspehi Matem. Nauk (N. S.), 2 (1947), pp.~60--106.

\bibitem{Krein-Krasnoselskii-Milman-1948}
{\sc M.~G. Kre\u{\i}, M.~A. Krasnosel$'$ski\u{\i}, and D.~Mil$'$man}, {\em
  {Concerning the deficiency numbers of linear operators in Banach space and
  some geometric questions}}, Sbornik Trudov Instit. Mat. Akad. Nauk. Ukr.
  S.S.R.,  (1948), pp.~97--112.

\bibitem{Liesen-Strakos-2003}
{\sc J.~Liesen and Z.~e. {Strako\v s}}, {\em {Krylov subspace methods}},
  {Numerical Mathematics and Scientific Computation}, Oxford University Press,
  Oxford, 2013.
\newblock Principles and analysis.

\bibitem{Munkres-Topology}
{\sc J.~R. Munkres}, {\em {Topology}}, Prentice Hall, Inc., Upper Saddle River,
  NJ, 2000.

\bibitem{Nemirovskiy-Polyak-1985}
{\sc A.~S. Nemirovskiy and B.~T. Polyak}, {\em {Iterative methods for solving
  linear ill-posed problems under precise information. {I}}}, Izv. Akad. Nauk
  SSSR Tekhn. Kibernet.,  (1984), pp.~13--25, 203.

\bibitem{Nemirovskiy-Polyak-1985-II}
\leavevmode\vrule height 2pt depth -1.6pt width 23pt, {\em {Iterative methods
  for solving linear ill-posed problems under precise information. {II}}},
  Engineering Cybernetics, 22 (1984), pp.~50--57.

\bibitem{Quarteroni-book_NumModelsDiffProb}
{\sc A.~Quarteroni}, {\em {Numerical models for differential problems}},
  vol.~16 of {MS\&A. Modeling, Simulation and Applications}, Springer, Cham,
  2017.
\newblock Third edition.

\bibitem{Saad-2003_IterativeMethods}
{\sc Y.~Saad}, {\em {Iterative methods for sparse linear systems}}, Society for
  Industrial and Applied Mathematics, Philadelphia, PA, second~ed., 2003.

\bibitem{Shkarin-2006-supecylic-shift}
{\sc S.~Shkarin}, {\em {A weighted bilateral shift with cyclic square is
  supercyclic}}, Bull. Lond. Math. Soc., 39 (2007), pp.~1029--1038.

\bibitem{Sifuentes-Embree-Morgan-2013}
{\sc J.~A. Sifuentes, M.~Embree, and R.~B. Morgan}, {\em {GMRES Convergence for
  Perturbed Coefficient Matrices, with Application to Approximate Deflation
  Preconditioning}}, SIAM Journal on Matrix Analysis and Applications, 34
  (2013), pp.~1066--1088.

\bibitem{Simoncini-Szyld-inexact-2003}
{\sc V.~Simoncini and D.~B. Szyld}, {\em {Theory of Inexact Krylov Subspace
  Methods and Applications to Scientific Computing}}, SIAM Journal on
  Scientific Computing, 25 (2003), pp.~454--477.

\bibitem{Simoncini-Szyld-inexact-2005}
\leavevmode\vrule height 2pt depth -1.6pt width 23pt, {\em {On the Occurrence
  of Superlinear Convergence of Exact and Inexact Krylov Subspace Methods}},
  SIAM Review, 47 (2005), pp.~247--272.

\bibitem{Tuzhilin-2020}
{\sc A.~A. Tuzhilin}, {\em {Lectures on Hausdorff and Gromov-Hausdorff Distance
  Geometry}}, arXiv:2012.00756 (2020).

\bibitem{Vandeneshof-2005}
{\sc J.~{van den Eshof}, G.~L. Sleijpen, and M.~B. {van Gijzen}}, {\em
  {Relaxation strategies for nested Krylov methods}}, Journal of Computational
  and Applied Mathematics, 177 (2005), pp.~347--365.

\bibitem{Winther-1980}
{\sc R.~Winther}, {\em {Some superlinear convergence results for the conjugate
  gradient method}}, SIAM J. Numer. Anal., 17 (1980), pp.~14--17.

\bibitem{Xue-Elman-2011}
{\sc F.~Xue and H.~C. Elman}, {\em {Fast inexact subspace iteration for
  generalized eigenvalue problems with spectral transformation}}, Linear
  Algebra and its Applications, 435 (2011), pp.~601--622.
\newblock Special Issue: Dedication to Pete Stewart on the occasion of his 70th
  birthday.

\bibitem{Zemke-2007}
{\sc J.-P.~M. Zemke}, {\em {Abstract perturbed Krylov methods}}, Linear Algebra
  and its Applications, 424 (2007), pp.~405--434.

\end{thebibliography}

\def\cprime{$'$}

\end{document}